\documentclass[11pt]{article}

\usepackage[margin=1in]{geometry} 
\usepackage{amsmath,amssymb,amsfonts,amsthm}
\usepackage{graphicx}
\usepackage{float}
\usepackage{xcolor}
\usepackage[authoryear, round]{natbib}
\usepackage{hyperref}
\usepackage{makecell}
\hypersetup{colorlinks,citecolor=blue,linkcolor=blue,breaklinks=true}
\usepackage{epstopdf}
\usepackage{yhmath} 
\usepackage{booktabs}
\usepackage{algorithm}
\usepackage{algpseudocode}
\usepackage{enumitem}
\usepackage{subcaption}
\usepackage{natbib}
\usepackage{ulem}
\allowdisplaybreaks
\usepackage{tikz}
\usetikzlibrary{shapes,arrows,positioning,backgrounds,shadows,calc}
\definecolor{myblue}{RGB}{43,87,154}
\definecolor{myorange}{RGB}{242,80,34}
\definecolor{mygreen}{RGB}{0,163,0}
\tikzstyle{arrow} = [thick,->,>=stealth]

\allowdisplaybreaks

\newtheorem{theorem}{Theorem}[section]
\newtheorem{lemma}{Lemma}[section]
\newtheorem{corollary}[theorem]{Corollary}
\newtheorem{proposition}[theorem]{Proposition}

\newtheorem{remark}{Remark}[section]
\newtheorem{assumption}{Assumption}[section]

\numberwithin{equation}{section}

\DeclareMathOperator*{\argmax}{arg\,max}

\begin{document}
	
\title{Boundary-aware Reinforcement Learning for Hypercube State Spaces via Deterministic Policy Gradient}
	
	\author{
		Lijun Bo
		\thanks{lijunbo@xidian.edu.cn, School of Mathematics and Statistics, 
        Xidian University, Xi'an, 710126, China}
		\and
		Yijie Huang
        \thanks{yijie.huang@polyu.edu.hk, Department of Applied Mathematics, The Hong Kong Polytechnic University, Kowloon, Hong Kong, China.}
        \and
        Chenhao Lu
		\thanks{luchenhao\_ustc@mail.ustc.edu.cn, School of Mathematical Sciences, University of Science and Technology of China, Hefei, 230026, China}
	}
	
	\date{}
	
	\maketitle

\begin{abstract}
 We develop a continuous-time deterministic policy gradient framework for reinforcement learning with reflected state dynamics, where the state process is governed by a controlled reflected stochastic differential equation on a hypercube. Under suitable regularity assumptions, we establish the connection between the value function and the Neumann Bellman equation, introduce an advantage-rate function that yields a deterministic policy gradient formula, and prove the martingale characterization theorem. Motivated by these theoretical results, we propose a continuous-time deep deterministic policy gradient algorithm for reflected stochastic systems, in which the Neumann boundary condition is imposed via either soft penalization or hard architectural constraint. We further quantify the discrepancy between the ideal continuous-time dynamics and the discretely sampled exploratory dynamics executed in practice, showing that the error decays as the time grid is refined and exploration noise vanishes. Our experiments on reservoir control problems illustrate the effectiveness of the RL framework, highlighting that boundary-aware methods substantially reduce Neumann boundary residuals and enhance learning stability.

\vspace{0.1in}
\end{abstract}

\section{Introduction}
\label{sec:introduction}

Reinforcement learning (RL) provides a general framework for sequential decision making in which an agent improves its policy through repeated interaction with an unknown environment. Classical RL has been developed predominantly in discrete time, where the environment is modeled as a Markov decision process and the agent updates value functions or policies from sampled transitions. Among the most influential value-based methods, Q-learning (see \cite{wat89} and \cite{wat92}) learns the optimal state-action value function in a model-free manner and has become a cornerstone of discrete-time RL. Temporal-difference learning (\cite{sut88}), SARSA (\cite{Rum94}), and actor-critic methods (\cite{kon00}, \cite{sut18}) further established the algorithmic foundation for modern RL. The combination of value-based RL with deep neural networks led to the deep Q-network (DQN) (\cite{mni15}), which demonstrated that high-dimensional sensory inputs can be handled by deep function approximation. Policy-gradient methods (\cite{will92} and \cite{sut00}) and their actor-critic variants have also become central in continuous-action control, where direct optimization of a parameterized policy is often more natural than maximization over a discrete action set.

Despite these successes, many real-world decision problems are not intrinsically discrete in time. Robotic manipulation, autonomous driving, high-frequency trading, queueing networks, energy storage, and water reservoir regulation all involve systems whose states evolve continuously and whose decisions may be updated at high frequency. Early attempts to bring RL into continuous time include advantage updating (\cite{ba94}), the continuous-time value-function and policy-improvement framework of \cite{doya00}, and the stochastic control perspective of \cite{muno97}. These works already emphasized that the continuous-time formulation should not be viewed as a mere limiting case of a fixed discrete-time algorithm. A central difficulty is that objects that are well behaved in discrete time may degenerate when the time step becomes small. In particular, \cite{tall19} showed that standard Q-learning-type methods can be highly sensitive to time discretization in near continuous-time environments, and that the classical discrete-time Q-function is not the right object in the continuous-time limit. This observation motivates the development of RL algorithms directly at the continuous-time level, with time discretization introduced only at the implementation stage.

A systematic stochastic-control formulation of continuous-time RL was introduced by \cite{wang20}, where exploration is modeled through entropy-regularized relaxed controls for diffusion processes. Building on this formulation, \cite{jia22a} developed a martingale approach to policy evaluation in continuous time and space, showing that the value function of a fixed policy can be characterized by the martingality of a suitable reward-adjusted process. The same martingale viewpoint was then used in \cite{jia22b} to derive continuous-time policy gradient and actor-critic algorithms. A major step toward value-based learning in continuous time was made by \cite{jia23}, who introduced the continuous-time counterpart of Q-learning. Since the usual 'big' Q-function collapses to the value function in continuous time, \cite{jia23} defined a first-order state-action object, the so-called 'little' q-function, which is closely related to the instantaneous advantage rate and the Hamiltonian. This q-function can be characterized by martingale conditions and used to design on-policy and off-policy actor-critic algorithms without discretizing the control problem in advance. The continuous-time RL literature has since expanded in several directions. Continuous-time q-learning has been extended to mean field control problems, where the state distribution enters the dynamics and reward, and the integrated and essential q-functions arise naturally (see \cite{wei24} and \cite{wei25}). More recently, \cite{ren26a} investigated mean field control with controlled common noise and the accompanying work \cite{ren26b} developed martingale-based q-learning and actor-critic algorithms for this formulation. Related martingale and q-learning ideas have also been developed for jump-diffusion models, both under Shannon entropy \cite{gao25} and under Tsallis entropy \cite{bo24b}. Another line of work studies the accuracy of discretely sampled stochastic policies, addressing the gap between continuous-time exploratory formulations and implementable piecewise-constant action processes (see \cite{jia25}). These developments show that continuous-time RL is not only a mathematical limit of discrete-time RL, but a distinct framework in which stochastic calculus, martingale methods, Hamilton-Jacobi-Bellman equations, and policy optimization interact in a fundamental way.

Most of the above continuous-time RL works are formulated with stochastic or relaxed policies. This is natural for exploration, but it also creates computational and conceptual difficulties. First, the relaxed policy is a measure-valued control, and its implementation therefore requires high-frequency sampling from the prescribed action distribution. Second, stochastic policies introduce an additional action-integration layer into policy evaluation and policy improvement. In continuous action spaces, expectations with respect to the policy distribution must generally be computed analytically or approximated by Monte Carlo sampling. The latter may increase the computational cost of the policy update, especially for high-dimensional action spaces. These considerations motivate deterministic policy gradient methods, which optimize a parameterized deterministic feedback policy. In discrete time, deterministic policy gradient methods (\cite{sil14}) provide an efficient alternative to stochastic policy gradients in continuous action spaces, since the policy gradient only integrates over the state distribution. The deep deterministic policy gradient algorithm (DDPG) by \cite{lil15} combines this idea with replay buffers and target networks, two stabilizing devices inherited from deep Q-learning. Recently, deterministic policy gradients have also been developed directly in continuous time. \cite{cheng25} derived a continuous-time deterministic policy gradient formula through the advantage-rate function and proposed a CT-DDPG algorithm based on martingale characterization. This framework has subsequently been extended to time-inconsistent stochastic control, where \cite{guo26} combined deterministic policy-gradient updates, martingale characterizations and inner fixed-point iterations to learn deterministic
equilibrium policies. In the mean-field direction, \cite{cheng26} developed a model-free deterministic actor-critic framework for extended mean-field control, in which both the dynamics and rewards may depend on the joint distribution of states and controls. This line of work is particularly relevant to applications in which deterministic feedback controls are ultimately intended for implementation.

The present paper studies a complementary issue that has received much less attention in continuous-time RL: state boundary. In many applications, the state process is constrained by some limits. For example, a reservoir has lower and upper storage bounds; an inventory system has capacity constraints; a queue length or storage process may be constrained to remain nonnegative; an engineering system may be required to stay inside a prescribed safe region. A natural mathematical model for such constrained dynamics is a reflected stochastic differential equation. When the state reaches the boundary, a finite-variation regulator pushes it back into the admissible domain. On a hypercube, this reflection acts componentwise on the lower and upper faces. The corresponding Bellman equation is therefore not an ordinary interior PDE, but a parabolic equation with Neumann boundary condition. The boundary condition determines how the critic should behave on reflecting faces, how the boundary regulator enters It\^o's formula, and how one should design stable value-function approximations near the boundary.

Motivated by these observations, this paper develops a boundary-aware actor-critic RL framework for reflected diffusion control on a hypercube. For a fixed policy, the value function satisfies a Neumann Bellman equation. We then use this structure to construct a continuous-time deep deterministic policy gradient method for reflected dynamics, which we call CT-DDPG-R. The main contributions of this paper are summarized as follows.
\begin{itemize}
    \item[(i)] For a fixed parameterized deterministic Markov policy, we characterize the corresponding value function through a parabolic Bellman equation on the interior of the hypercube together with Neumann conditions on its reflecting faces. We derive a continuous-time deterministic policy gradient formula. A key feature of the analysis is that the regulator terms disappear from the final gradient expression because the value function satisfies the Neumann boundary condition. The result extends continuous-time deterministic policy gradient theory from unconstrained diffusions to state-constrained reflected systems. Then, we introduce the continuous-time advantage-rate function associated with a deterministic policy and establish a martingale characterization of both the value function and the advantage-rate function. 
     \item[(ii)] Guided by the theoretical results, we propose CT-DDPG-R, a boundary-aware continuous-time deterministic policy algorithm. The method combines martingale-based critic learning, deterministic actor updates, experience replay, target networks, and two treatments of the Neumann condition: a soft Neumann penalty in the critic loss and a hard Neumann architecture based on cosine Fourier features. We evaluate the proposed method on one-, two-, and five-dimensional reservoir regulation problems. The experiments demonstrate that incorporating the Neumann condition substantially reduces boundary residuals and improves training stability, with the advantages becoming more pronounced as the state dimension increases.
     \item[(iii)] The theoretical target policy is a continuous-time deterministic feedback, whereas the actions executed during training are sampled on a discrete time grid and perturbed by exploration noise. We quantify the discrepancy between the reflected state process under the ideal target policy and the process generated by the implementable piecewise-constant exploratory actions. The resulting estimate provides a theoretical justification for the discrete-time data-collection procedure used in CT-DDPG-R.
\end{itemize}

The rest of the paper is organized as follows. Section \ref{sec:model} introduces the controlled reflected SDE on a hypercube and the objective functional. Section \ref{sec:neumann} derives the Neumann Bellman equation for a fixed deterministic policy and connects it with the value function. Section \ref{sec:dpg-m} establishes the deterministic policy gradient formula and develops the martingale characterizations used for model-free learning. Section \ref{sec:algorithm} presents the CT-DDPG-R algorithm with the soft and hard constraints, and estimates the error of the state process in real environment. Section \ref{sec:application} applies our algorithm to the reservoir control problems and reports its performance.

\section{Problem Formulation}\label{sec:model}

In this section, we formulate the controlled reflected diffusion problem on a hypercube. After specifying the reflected state dynamics and admissible policies, we define the value function associated with a deterministic Markov policy. This setup provides the stochastic control framework on which the subsequent Neumann Bellman equation, martingale characterization, and deterministic policy gradient are built.

Let $n\geq1$ and $-\infty<\underline{x}_i<\overline{x}_i<\infty$ for $i=1,\ldots,n$. Consider a hypercube given by $ Q:=\prod_{i=1}^n [\underline{x}_i,\overline{x}_i]\subset\mathbb{R}^n$ whose inner is $Q^\circ:=\prod_{i=1}^n (\underline{x}_i,\overline{x}_i)$. We then introduce the following reflected dynamics on the hypercube given by $Q$, which is given by, for $t\in(0,T]$ with $T>0$,
\begin{align}\label{eq:cube-rsde}
X_t=\xi_0+\int_0^t b(s,X_s,a_s)ds+\int_0^t\sigma(s,X_s,a_s)dW_s+K_t^--K_t^+\in Q,
\end{align}
where $W=(W_t)_{t\in[0,T]}$ is an $m$-dimensional Brownian motion defined on the filtered probability space $(\Omega,\mathcal{F},\mathbb{P},\mathbb{F})$ with the filtration $\mathbb{F}=(\mathcal{F}_t)_{t\in[0,T]}$ satisfying the usual conditions, and $\xi_0\in Q$ is the square-integrable initial data independent of $W$. Here, $a=(a_t)_{t\in[0,T]}$ is the $\mathbb{F}$-adapted control strategy which takes values on a convex and compact action set $A\subset\mathbb{R}^d$, $b:[0,T]\times\mathbb{R}^n\times A\to\mathbb{R}^n$ and $\sigma:[0,T]\times\mathbb{R}^n\times A\to\mathbb{R}^{n\times m}$ are coefficients of reflected dynamics. In particular, $K^\pm=(K_t^{\pm,1},\ldots,K_t^{\pm,n})_{t\in[0,T]}$ are continuous and componentwisely non-decreasing processes satisfying $K_0^\pm=0$ and for any $i=1,\dots,n$,
\begin{align*}
\int_0^T \mathbf 1_{\{X_t^i>\underline{x}_i\}}\,dK_t^{-,i}=0,\quad \int_0^T \mathbf 1_{\{X_t^i<\overline{x}_i\}}\,dK_t^{+,i}=0.    
\end{align*}
The $i$-th element of $K^\pm$ is pushed only when it hits the corresponding lower or upper face. Equation \eqref{eq:cube-rsde} can be understood in the standard Skorokhod sense on $Q$ (see \cite{pi14}).

In contrast to the stochastic policies commonly adopted in prior continuous-time RL works (e.g., \cite{jia23}), which require sampling from action distributions and integrating over continuous action spaces, we adopt a deterministic Markov policy framework. Under this setting, the control action at time $t$ and state $x$ is prescribed by a deterministic feedback law of the form $a_t = \mu(t, X_t)$, with $\mu:[0,T]\times Q \to A$. To enable flexible function approximation and integration with deep learning, we introduce a parameterized family of deterministic policies $\mu_\phi: [0,T]\times \mathbb{R}^n \to A$, indexed by the parameter vector $\phi\in\mathbb{R}^k$.

Let $r:[0,T]\times Q\times A\to\mathbb{R}$ be the running reward and $g:Q\to\mathbb{R}$ be the terminal reward. Now,  we fix a discount factor $\beta > 0$. For any policy parameter $\phi \in \mathbb{R}^k$, let $X^{\phi}=(X_t^{\phi})_{t\in[0,T]}$ be the reflected state process satisfying \eqref{eq:cube-rsde} with the Markovian feedback control $a_t=\mu_\phi(t,X_t^{\phi})$ for $t\in[0,T]$. The performance of this policy is measured by the following dynamical expected cumulative return given by, for $(t,x)\in[0,T]\times Q$ and $\phi\in\mathbb{R}^k$,
\begin{align}
\label{eq:cube-value}
V^{\phi}(t,x):=\mathbb{E}\left[\int_t^T e^{-\beta (s-t)} r\left(s,X_s^{\phi},\mu_\phi(s,X_s^{\phi})\right)ds+e^{-\beta( T-t)}g(X_T^{\phi})\big|X_t^\phi=x\right],    
\end{align}
and we aim to maximize $V^{\phi}(t,x)$ over $\phi \in \mathbb{R}^k$.

\section{Neumann Bellman Equation on Hypercube}
\label{sec:neumann}
In this section, we characterize the value function \eqref{eq:cube-value} through a Neumann Bellman equation on the hypercube. Before that, several assumptions are imposed to ensure that our problem is well posed and that the associated value function is well defined.

\begin{assumption}[Regularity of Coefficients]\label{ass:coefficient}
The functions $b(t,x,a)$, $\sigma(t,x,a)$ and $r(t,x,a)$ are continuous in $(t,x,a)$. Moreover, there exists a constant $C >0$ such that, for any $(t,x,a),(t,x',a')\in[0,T]\times Q\times A$,
\begin{align*}
|b(t,x,a)-b(t,x',a')|+|\sigma(t,x,a)-\sigma(t,x',a')|&\le C\left(|x-x'|+|a-a'|\right),\nonumber\\
|b(t,x,a)|+|\sigma(t,x,a)| &\le C\left(1+|x|+|a|\right),\nonumber\\
|r(t,x,a)-r(t,x',a')|+|g(x)-g(x')| &\le C \left(|x-x'|+|a-a'|\right), \nonumber\\
|r(t,x,a)|+|g(x)| &\leq C\left(1+|x|^2+|a|^2\right).
\end{align*}
\end{assumption}

\begin{assumption}[Regularity of Feedback Policy]\label{ass:policy}
The map $(\phi,t,x)\mapsto \mu_\phi(t,x)$ is continuous on
$\mathbb{R}^k\times[0,T]\times Q$. Moreover, for any compact parameter set $\Phi\subset\mathbb{R}^k$,
there exists $L_\Phi>0$ such that $|\mu_\phi(t,x)-\mu_\phi(t',x')|\le L_\Phi (|t-t'|+|x-x'|)$ for any $(\phi,t,t',x,x')\in \Phi\times[0,T]^2\times Q^2$.
\end{assumption}

\begin{remark}
Assumptions \ref{ass:coefficient}, \ref{ass:policy} yield that the closed-loop coefficients $b_\phi(t,x):=b(t,x,\mu_\phi(t,x))$ and $\sigma_\phi(t,x):=\sigma(t,x,\mu_\phi(t,x))$ are Lipschitz in $x\in Q$. Since $Q$ satisfies the uniform exterior sphere condition and the uniform cone condition, it follows from Theorem 2.5.4 in \cite{pi14} that reflected stochastic differential equation (RSDE) \eqref{eq:cube-rsde} has a unique strong solution for each fixed $\phi$.
\end{remark}

\begin{assumption}[Uniform Ellipticity]
\label{ass:ellipticity}
The diffusion coefficient $\sigma(t,x,a)$ satisfies the uniform ellipticity condition, i.e., there exists a constant $\kappa>0$ such that $\xi^\top \Sigma(t,x,a)\,\xi \ge \kappa |\xi|^2$ for $(t,x,a)\in[0,T]\times Q \times A$ and $\xi\in\mathbb{R}^n$. Here, $\Sigma(t,x,a):=\sigma(t,x,a)\sigma(t,x,a)^\top$.
\end{assumption}

To support the subsequent theoretical derivations of the deterministic policy gradient and the martingale characterization, we impose the following regularity assumption on the value function \eqref{eq:cube-value} :
\begin{assumption}
\label{ass:value}
For each fixed policy parameter $\phi$, the value function $V^\phi \in C^{1,2}([0,T) \times Q^{\circ}) \cap C([0,T] \times Q)$ with $\nabla_xV^{\phi} \in C([0,T) \times Q)$. Furthermore, its 1st and 2nd spatial derivatives are absolutely integrable over $[0,T)$ for any $x\in Q$, i.e., there exists a function $\Psi\in L^1([0,T))$ such that
\begin{align*}
\sup_{x \in Q}\left( |\nabla_x V^{\phi}(t,x)|+|\nabla^2_x V^{\phi}(t,x)| \right) \leq \Psi(t),\quad \forall t\in[0,T].
\end{align*}
\end{assumption}

\begin{remark}
For reflected diffusion control problems defined on hypercubes, the value function is conventionally characterized via standard dynamic programming and viscosity solution theory as the unique continuous viscosity solution to the corresponding Neumann boundary value problem \eqref{eq:neumann-pde}. Nevertheless, to the best of our knowledge, a general classical regularity theorem for such reflected hypercube control problems remains absent in the existing literature. In this work, we impose the regularity condition $V^\phi\in C^{1,2}([0,T)\times Q^{\circ}) \cap C([0,T]\times Q)$ with $\nabla_x V^{\phi} \in C([0,T) \times Q)$, instead of the stronger regularity requirement $V^\phi\in C^{1,2}([0,T]\times Q)$. This distinction is critical for problem formulations on hypercube domains. Specifically, compatibility conditions, which connect the terminal condition and the Neumann boundary condition, may be violated. As a result, full classical regularity over the closed cylindrical domain $[0,T]\times Q$ cannot be generally guaranteed. The regularity assumption adopted herein aligns with the classical regularity results established in \cite{bo24a} for Robin boundary value problems on the nonnegative orthant.
\end{remark}

Before deriving the Neumann Bellman equation, we record a basic property of reflected diffusions that will be used repeatedly below. Although the finite-variation regulators $K^-$ and $K^+$ accumulate when the process hits the boundary, the state process itself does not spend positive Lebesgue time on the boundary. In other words, the reflection is carried by boundary local time, which is singular with respect to the time variable (c.f. Proposition 6.1 in \cite{kang14}).
\begin{lemma}\label{lem:boundary-zero-time}
Under Assumptions \ref{ass:coefficient}, \ref{ass:policy} and \ref{ass:ellipticity}, let $X^\phi=(X^{\phi}_t)_{t\in[0,T]}$ be a strong solution to the RSDE \eqref{eq:cube-rsde} under $\mu_\phi$. Then $\int_0^T \mathbf {1}_{\partial Q}(X_s^\phi)ds = 0$, a.s..
\end{lemma}

We define the controlled generator associated with a fixed action $a\in A$ by, for $\varphi\in C^2([0,T]\times Q)$,
\begin{align*}
\mathcal{L}^a[\varphi](t,x):=\partial_t \varphi(t,x)+b(t,x,a)^\top \nabla_x\varphi(t,x)+\frac{1}{2} \mathrm{Tr}[\Sigma(t,x,a)\nabla_x^2\varphi(t,x)]-\beta \varphi(t,x).    
\end{align*}
Then, we have the following result on the HJB equation for reflected diffusion processes under deterministic feedback policies.
\begin{lemma}\label{lem:verification}
Let Assumptions \ref{ass:coefficient}, \ref{ass:policy}, \ref{ass:ellipticity} and \ref{ass:value} hold. Then, the value function $V^\phi$ defined by \eqref{eq:cube-value} satisfies the following Neumann boundary value problem given by
\begin{align}\label{eq:neumann-pde}
\begin{cases}
\displaystyle \mathcal L^{\mu_\phi(t,x)}[V^\phi](t,x)+r(t,x,\mu_\phi(t,x))=0,~~\forall (t,x)\in[0,T)\times Q^\circ,\\
\displaystyle \partial_{x_i}V^\phi(t,x)=0,~~\forall (t,x)\in[0,T)\times \partial Q,~ i\in I_-(x)\cup I_+(x),\\
\displaystyle V^\phi(T,x)=g(x),~~\forall x\in Q,
\end{cases}
\end{align}
where the index sets $I_-(x)$ and $I_+(x)$ are respectively defined by $I_-(x):=\{i\in\{1,\dots,n\}:~x_i=\underline x_i\}$ and $I_+(x):=\{i\in\{1,\dots,n\}:~x_i=\overline x_i\}$. Conversely, if a function $V \in C^{1,2}([0,T)\times Q^{\circ})\cap C([0,T]\times Q)$ satisfying $\nabla_x V \in C([0,T) \times Q)$ is a classical solution of \eqref{eq:neumann-pde}, then it admits the probabilistic representation \eqref{eq:cube-value}.
\end{lemma}

\begin{proof}
Let $(t,x)\in[0,T]\times Q$ and write $X_s=X_s^{t,x,\phi}$ for $s\in[t,T]$. We first prove the interior equation in \eqref{eq:neumann-pde}.  Fix $(t,x)\in[0,T)\times Q^\circ$. We choose $\rho>0$ such that $\overline{B(x,\rho)}\subset Q^\circ$, where $B(x,\rho):=\{y:~|y-x| < \rho\}$. Given $h>0$, define the stopping time $\tau_h:=\inf\{s\ge t:~X_s\notin B(x,\rho)\}\wedge (t+h)\wedge T$. For $s\in [t,\tau_h]$, no reflection occurs for $X_s$, and hence $\mathrm{d} K_s^{-,i}=\mathrm{d} K_s^{+,i}=0$ for $i=1,\ldots,n$. Applying It\^o's formula to $e^{-\beta(s-t)}V^\phi(s,X_s)$ from $t$ to $\tau_h$,
and using the dynamic programming principle (DPP), we obtain
\begin{align}\label{eq:Ito-V}
0=\mathbb{E}\left[\int_t^{\tau_h}e^{-\beta(s-t)}\left(\mathcal L^{\mu_\phi(s,X_s)}[V^\phi](s,X_s)+r\left(s,X_s,\mu_\phi(s,X_s)\right)\right)\mathrm{d}s\right].
\end{align}
Dividing both sides of  of \eqref{eq:Ito-V} by $h$ and letting $h\downarrow 0$, the dominated convergence theorem (DCT), together with the continuity of the coefficients and of the derivatives of $V^\phi$ in $Q^\circ$, yields that
\begin{align*}
\mathcal{L}^{\mu_\phi(t,x)}[V^\phi](t,x)+r(t,x,\mu_\phi(t,x))=0,\quad(t,x)\in[0,T)\times Q^\circ .   
\end{align*}
Substituting $t=T$ into \eqref{eq:cube-value} gives the terminal condition $V^\phi(T,x)=g(x)$ for all $x\in Q$.

Next, we prove the Neumann boundary condition in \eqref{eq:neumann-pde}. Fix $i\in\{1,\ldots,n\}$ and $(t,x)\in[0,T)\times \mathrm{ri}(\Gamma_i^-)$,
where $\mathrm{ri}(\Gamma_i^-)$ denotes the relative interior of the face
$\Gamma_i^-=\{x\in Q:~\ x_i=\underline{x}_i\}$.
We choose a neighborhood $U$ of $x$ such that
$U\cap \partial Q\subset \Gamma_i^-$. Given $h>0$, let $\tau_h:=\inf\{s\ge t:~X_s\notin U\}\wedge (t+h)\wedge T$. On \([t,\tau_h]\), the only possible regulator term is $K^{-,i}$.
By DPP, the interior equation and Lemma \ref{lem:boundary-zero-time} yield that
\begin{align}
\label{eq:lower-face-integral-zero-new}
\mathbb{E}\left[\int_t^{\tau_h}e^{-\beta(s-t)}\partial_{x_i}V^\phi(s,X_s) \mathrm{d} K_s^{-,i}\right]=0.    
\end{align}
We now show that $\partial_{x_i}V^\phi(t,x)=0$ by contradiction. Suppose first that $\partial_{x_i}V^\phi(t,x)>0$.
By continuity of $\partial_{x_i}V^\phi$, we may assume that $\partial_{x_i}V^\phi(\tilde{t},\tilde{x})\ge \varepsilon>0$, for all $(\tilde{t},\tilde{x})\in [t,t+h]\times (U\cap \Gamma_i^-)$. By Assumption \ref{ass:ellipticity}, the diffusion is uniformly nondegenerate. Since the reflected state process starts from the boundary face, the regulator $K^{-,i}$ increases with positive probability before $\tau_h$, that is $\mathbb{P}(K_{\tau_h}^{-,i}-K_t^{-,i}>0)>0$. Hence, we have
\begin{align*}
\mathbb{E}\left[\int_t^{\tau_h}e^{-\beta(s-t)}\partial_{x_i}V^\phi(s,X_s)\mathrm{d}K_s^{-,i}\right] \geq\varepsilon
\mathbb{E}\left[\int_t^{\tau_h}e^{-\beta(s-t)}\mathrm{d} K_s^{-,i} \right]>0,
\end{align*}
which contradicts \eqref{eq:lower-face-integral-zero-new}. The case $\partial_{x_i}V^\phi(t,x)<0$ leads to an analogous contradiction. Hence, we have $\partial_{x_i}V^\phi(t,x)=0$ for all $(t,x)\in[0,T)\times \mathrm{ri}(\Gamma_i^-)$. The proof for the upper face $\Gamma_i^+:=\{x\in Q:~x_i=\overline{x}_i\}$ is analogous, and applying the same contradiction argument yields that $\partial_{x_i}V^\phi(t,x)=0$ for all $(t,x)\in[0,T)\times \mathrm{ri}(\Gamma_i^+)$. Thus we can extend the result to the whole $[0,T) \times \partial Q$ by the continuity of $\nabla_x V^{\phi} \in C([0,T)\times Q)$.

Conversely, suppose that $V \in C^{1,2}([0,T)\times Q^{\circ})\cap C([0,T]\times Q)$ with $\nabla_x V \in C([0,T) \times Q)$ be a classical solution of \eqref{eq:neumann-pde}. Fix $(t,x)\in[0,T)\times Q$ and set $\varepsilon \in (0,T-t)$. For $s\in[t,T-\varepsilon]$, we define
\begin{align}\label{eq:Y-martingale}
Y_s:=e^{-\beta(s-t)}V(s,X_s^{\phi})+\int_t^s e^{-\beta(u-t)}r\left(u,X_u^{\phi},\mu_\phi(u,X_u^{\phi})\right)du.
\end{align}
It follows from It\^o's rule that 
\begin{align*}
\mathrm{d}Y_s&=e^{-\beta(s-t)}\left(\mathcal{L}^{\mu_\phi(s,X_s^{\phi})}[V](s,X_s^{\phi})+r\left(s,X_s^{\phi},\mu_\phi(s,X_s^{\phi})\right)\right)\mathrm{d} s \nonumber\\
&\quad+e^{-\beta(s-t)}\nabla_x V(s,X_s^{\phi})^\top\sigma\left(s,X_s^{\phi},\mu_\phi(s,X_s^{\phi})\right)\mathrm{d}W_s \nonumber\\
&\quad+e^{-\beta(s-t)}\sum_{i=1}^n\partial_{x_i}V(s,X_s^{\phi})\mathrm{d} K_s^{-,i}-e^{-\beta(s-t)}\sum_{i=1}^n\partial_{x_i}V(s,X_s^{\phi})\mathrm{d}K_s^{+,i}.
\end{align*}
The $\mathrm{d}K^{\pm,i}$ terms vanish by the Neumann boundary condition, and 
the drift term vanishes by the interior equation in
\eqref{eq:neumann-pde} and Lemma \ref{lem:boundary-zero-time}. Consequently, $Y=(Y_s)_{s\in[t,T-\varepsilon]}$ is a martingale. Applying the martingality on the right hand of \eqref{eq:Y-martingale}, it follows that
\begin{align*}
V(t,x)=\mathbb{E}_{t,x}\left[\int_t^{T-\varepsilon}e^{-\beta(s-t)}r\left(s,X_s^{\phi},\mu_\phi(s,X_s^{\phi})\right)\mathrm{d}s+e^{-\beta(T-\varepsilon-t)}
V(T-\varepsilon,X_{T-\varepsilon}^{\phi})\right].
\end{align*}
Let $\varepsilon\downarrow0$ and then the continuity of $V \in C([0,T] \times Q)$, the terminal condition $V(T,\cdot)=g(\cdot)$ and the boundedness of the compact set $[0,T]\times Q$ jointly imply that
\begin{align*}
V(t,x)=\mathbb{E}_{t,x}\left[\int_t^Te^{-\beta(s-t)}
r\left(s,X_s^{\phi},\mu_\phi(s,X_s^{\phi})\right)\mathrm{d}s+e^{-\beta(T-t)}g(X_T^{\phi})\right],   
\end{align*}
which is precisely \eqref{eq:cube-value}. Thus, the proof of the lemma is complete. 
\end{proof}

\section{Policy Gradient and Martingales on Hypercube}\label{sec:dpg-m}

In this section, we establish two theoretical foundations for reflected diffusion control on the hypercube: a deterministic policy gradient proposition and a martingale characterization. The first result follows the line of deterministic policy gradient methods developed in \cite{sil14} and \cite{cheng25}, while the second extends the martingale approach of \cite{jia23} for controlled diffusions to the reflected state setting. These two results correspond, respectively, to the policy-improvement and policy-evaluation components of the actor-critic algorithm developed in the next section. In particular, the deterministic policy gradient formula provides the direction for updating the actor, whereas the martingale characterization yields model-free learning conditions for the value function and the advantage-rate function.

\subsection{Deterministic policy gradient on the hypercube}
\label{sec:dpg}
For each $\phi\in \mathbb R^k$, we define the continuous-time advantage rate function by
\begin{align}\label{eq:q-function}
q_\phi(t,x,a):=\mathcal L^a[V^\phi](t,x)+r(t,x,a),\quad \forall  (t,x,a)\in[0,T)\times Q\times A.   
\end{align}
Then, we have from Lemma \ref{lem:verification} that
\begin{align}\label{eq:adv-zero}
q_\phi\left(t,x,\mu_\phi(t,x)\right)=0,\quad \forall (t,x)\in[0,T)\times Q^{\circ}.
\end{align}

To establish the deterministic policy gradient, we need additional differentiability of the coefficients and feedback policy with respect to the action variable and the policy parameter.

\begin{assumption}
\label{ass:action-diff}
The functions $\nabla_a b$, $\nabla_a \Sigma$, and $\nabla_a r$ exist and are continuous on $[0,T]\times Q\times A$.
Moreover, the partial derivative $(\phi,t,x)\mapsto \nabla_\phi \mu_\phi(t,x)$ is continuous on $\mathbb{R}^k\times[0,T]\times Q$.
\end{assumption}
Fix $\phi\in \mathbb{R}^k$. By using the differentiability from Assumption \ref{ass:action-diff}, we can take the derivative with respect to $\phi$ on the both side of \eqref{eq:q-function}, and hence, for $(t,x,a)\in[0,T)\times Q\times A$,
\begin{align}\label{eq:nabla-q-funtion}
\nabla_a q_\phi(t,x,a)=
(\nabla_a b(t,x,a))^\top \nabla_x V^\phi(t,x)
+\frac{1}{2}\mathrm{Tr}((\nabla_a\Sigma(t,x,a))\nabla_x^2V^\phi(t,x))+\nabla_a r(t,x,a).   
\end{align}
We also need to derive a stability property of the closed-loop system under parameter perturbations. The following lemma asserts that small changes in the policy parameter $\phi$ result in corresponding small changes in the resulting state trajectory in the mean-square sense, uniformly over the entire time horizon.
\begin{lemma}
\label{lem:cube-stability}
Let Assumptions \ref{ass:coefficient} and \ref{ass:policy} hold. Then, for any $(t,x)\in[0,T]\times Q$ and $h\in\mathbb R^k$, it holds that
\begin{align*}
\lim_{\varepsilon\to0}
\mathbb{E}\left[\sup_{s\in[t,T]}\left|X_s^{t,x,\phi+\varepsilon h}-X_s^{t,x,\phi}\right|^2\right]=0.    
\end{align*}
\end{lemma}

\begin{proof}
We write $X_u^\varepsilon:=X_u^{t,x,\phi+\varepsilon h}$ and $X_u:=X_u^{t,x,\phi}$, and define the corresponding closed-loop coefficients $ b_\varepsilon(u,y):=b(u,y,\mu_{\phi+\varepsilon h}(u,y))$, $\sigma_\varepsilon(u,y):=\sigma(u,y,\mu_{\phi+\varepsilon h}(u,y))$, $b_0(u,y):=b(u,y,\mu_{\phi}(u,y))$ and $\sigma_0(u,y):=\sigma(u,y,\mu_{\phi}(u,y))$. Let $Y^\varepsilon=(Y^\varepsilon_u)_{u\in[t,T]}$ and $Y=(Y_u)_{u\in[t,T]}$ satisfy the following dynamics without reflection:
\begin{align*}
Y_u^\varepsilon=x+\int_t^u b_\varepsilon(r,X_r^\varepsilon)\mathrm{d}r+\int_t^u \sigma_\varepsilon(r,X_r^\varepsilon)\mathrm{d}W_r,\quad Y_u=x+\int_t^u b_0(r,X_r)\mathrm{d}r+\int_t^u \sigma_0(r,X_r)\mathrm{d}W_r.    
\end{align*}
Then, the reflected processes $X^{\varepsilon}=(X^{\varepsilon}_t)_{t\in[0,T]}$ and $X=(X_t)_{t\in[0,T]}$ can be represented via a so-called Skorokhod map $\Gamma_{Q}:C([0,T];\mathbb{R}^n)\to C([0,T];Q)$ on the hypercube $Q$ by $X^\varepsilon=\Gamma_Q(Y^\varepsilon)$ and $X=\Gamma_Q(Y)$, which means that $X^{\varepsilon},X$ are solutions respectively to the Skorokhod problems $(Q,Y^{\varepsilon})$ and $(Q,Y)$. Moreover, the Skorokhod map $\Gamma_{Q}$ is Lipschitz in the sup norm (c.f. \cite{du91}). Consider $F_\varepsilon(u):=
\mathbb{E}[\sup_{s\in[t,u]}|X_s^\varepsilon-X_s|^2]$ for $u\geq t$ and we have $F_\varepsilon(u)\leq C\mathbb{E}[\sup_{s\in[t,u]}|Y_s^\varepsilon-Y_s|^2]$ for some constant $C=C(T)$. From Jensen's inequality and Burkholder-Davis-Gundy Inequality (BDG Inequality), it follows that
\begin{align*}
\mathbb{E}\left[\sup_{s\in[t,u]}\left|Y_s^\varepsilon-Y_s\right|^2\right]
&\le C\int_t^u
\mathbb{E}\left[\left|b_\varepsilon(r,X_r^\varepsilon)-b_0(r,X_r)\right|^2\right]\mathrm{d}r+C\int_t^u
\mathbb{E}\left[\left|\sigma_\varepsilon(r,X_r^\varepsilon)-\sigma_0(r,X_r)\right|^2\right]\mathrm{d}r.
\end{align*}
Note that $|b_\varepsilon(r,X_r^\varepsilon)-b_0(r,X_r)|
\leq|b_\varepsilon(r,X_r^\varepsilon)-b_\varepsilon(r,X_r)|
+|b_\varepsilon(r,X_r)-b_0(r,X_r)|$. Since $b$ is Lipschitz in $(x,a)$, the first term $|b_\varepsilon(r,X_r^\varepsilon)-b_\varepsilon(r,X_r)|\le C(|X_r^\varepsilon-X_r|+|\mu_{\phi+\varepsilon h}(r,X_r^{\varepsilon})-\mu_{\phi+\varepsilon h}(r,X_r)|)$. We can choose $|h| \leq 1, \varepsilon \leq 1$ so that, by Assumption \ref{ass:policy}, there exists $L>0$ such that $|\mu_{\tilde{\phi}}(r,x)-\mu_{\tilde{\phi}}(r,x')| \leq L|x-x'|$ for $(\tilde{\phi},r,x,x') \in \{\tilde{\phi}:~|\tilde{\phi}-\phi|\leq 1\} \times [0,T] \times Q^2$, which yields $|b_\varepsilon(r,X_r^\varepsilon)-b_\varepsilon(r,X_r)|\le C(|X_r^\varepsilon-X_r|)$. Here, the constant $C$ depends on the compact set of parameters containing both $\phi$ and $\phi+\varepsilon h$. For the second term, we consider 
$\delta_\varepsilon:=
\sup_{(u,y)\in[t,T]\times Q}|b_\varepsilon(u,y)-b_0(u,y)|+
\sup_{(u,y)\in[t,T]\times Q}|\sigma_\varepsilon(u,y)-\sigma_0(u,y)|$ which tends to zero since $Q$ is compact and the mapping $(\phi,u,y)\mapsto \mu_\phi(u,y)$ is continuous. Then $|b_\varepsilon(r,X_r)-b_0(r,X_r)|\le\sup_{(u,y)\in[t,T]\times Q}|b_\varepsilon(u,y)-b_0(u,y)|\le \delta_\varepsilon$. Combining these estimates, we obtain that $|b_\varepsilon(r,X_r^\varepsilon)-b_0(r,X_r)|\le C|X_r^\varepsilon-X_r|+\delta_\varepsilon$. A similar argument for the diffusion coefficient gives that
$|\sigma_\varepsilon(r,X_r^\varepsilon)-\sigma_0(r,X_r)|\le C|X_r^\varepsilon-X_r|+\delta_\varepsilon$. Consequently, we obtain
\begin{align*}
\mathbb{E}\left[\sup_{s\in[t,u]}\left|Y_s^\varepsilon-Y_s\right|^2\right] \leq C \int_t^u \mathbb{E} \left[ \left| X^{\varepsilon}_r -X_r \right|^2 \right] \mathrm{d} r + C \delta_{\varepsilon}^2.    
\end{align*}
This implies that $F_\varepsilon(u)\leq C\int_t^u F_\varepsilon(r)\mathrm{d}r + C\delta_\varepsilon^2$. It follows from Gronwall's inequality that $F_\varepsilon(T)\le C\delta_\varepsilon^2 e^{C(T-t)}\to0$ as $\varepsilon \to 0$. Thus, the desired $L^2$-uniform stability is established.
\end{proof}

We now present the main result of this section, which gives the deterministic policy gradient for reflected diffusions on a hypercube. 

\begin{proposition}
\label{prop:cube-dpg}
Let assumptions of Lemma \ref{lem:verification} and Assumption \ref{ass:action-diff} hold.
Then, for any $(t,x)\in[0,T]\times Q$, the mapping
$\phi\mapsto V^\phi(t,x)$ is G\^ateaux differentiable. More precisely, for
any $h\in\mathbb R^k$,
\begin{align*}
\lim_{\varepsilon\to0}
\frac{V^{\phi+\varepsilon h}(t,x)-V^\phi(t,x)}{\varepsilon}
=\mathbb{E}\left[\int_t^T e^{-\beta(s-t)}
(\nabla_\phi\mu_\phi(s,X_s^{\phi})h)^\top\nabla_a q_\phi(s,X_s^{\phi},\mu_\phi(s,X_s^{\phi}))\mathrm{d}s\bigg|X_t^{\phi}=x\right].
\end{align*}
This equivalently yields that
\begin{align}\label{eq:cube-dpg}
\nabla_{\phi}V^\phi(t,x)=\mathbb{E}\left[\int_t^T e^{-\beta(s-t)}
(\nabla_\phi\mu_\phi(s,X_s^{\phi}))^\top\nabla_a q_\phi(s,X_s^{\phi},\mu_\phi(s,X_s^{\phi}))ds\bigg|X_t^{\phi}=x\right].
\end{align}
\end{proposition}

\begin{proof}
The case $t=T$ is immediate, since the terminal condition $V^\phi(T,x)=g(x)$ does not depend on
$\phi$. We therefore fix $(t,x)\in[0,T)\times Q$ and $h\in\mathbb{R}^k$, and consider a perturbed policy parameter $\phi+\varepsilon h$ for $\varepsilon\in\mathbb{R}/\{0\}$. For any $\eta\in(0,T-t)$, by applying the It\^o's rule to $e^{-\beta(s-t)}V^\phi(s,X_s^{\phi + \varepsilon h})$ on $[t,T-\eta]$, we have
\begin{align*}
&\mathrm{d}\left(e^{-\beta(s-t)}V^\phi(s,X^{\phi + \varepsilon h}_s)\right)
= e^{-\beta(s-t)}\mathcal{L}^{\mu_{\phi+\varepsilon h}(s,X^{\phi +\varepsilon h}_s)}[V^\phi](s,X^{\phi +\varepsilon h}_s)\mathrm{d}s \nonumber
\\
&\qquad+ e^{-\beta(s-t)}\nabla_x V^{\phi}(s,X^{\phi +\varepsilon h}_s)^\top \sigma(s,X_s^{\phi + \varepsilon h},\mu_{\phi+\varepsilon h}(s,X_s^{\phi + \varepsilon h}))\mathrm{d}W_s\\
&\qquad + e^{-\beta(s-t)} \sum_{i=1}^n \partial_{x_i} V^\phi(s,X^{\phi +\varepsilon h}_s) \mathrm{d} K^{-,i,\phi+ \varepsilon h}_s - e^{-\beta(s-t)} \sum_{i=1}^n \partial_{x_i} V^\phi(s,X^{\phi + \varepsilon h}_s)\mathrm{d} K^{+,i,\phi + \varepsilon h}_s. \nonumber
\end{align*}
Using the Neumann boundary condition in \eqref{eq:neumann-pde}, the reflection terms vanish. Integrating over $[t,T-\eta]$ and taking expectations given $X^{\phi+\varepsilon h}_t=x$, we obtain
\begin{align*}
V^\phi(t,x)&=\mathbb{E}\bigg[e^{-\beta(T-\eta-t)}V^\phi(T-\eta,X_{T-\eta}^{\phi+\varepsilon h})
\\
&\qquad -\int_t^{T-\eta} e^{-\beta(s-t)}\mathcal{L}^{\mu_{\phi+\varepsilon h}(s,X_s^{\phi + \varepsilon h})}[V^\phi](s,X_s^{\phi + \varepsilon h})\mathrm{d}s\bigg| X_t^{\phi + \varepsilon h}=x\bigg].
\end{align*}
Since $V^\phi\in C([0,T]\times Q)$ and $Q$ is compact, letting
$\eta \downarrow 0$, and the DCT yields that
\begin{align}\label{eq:val-under-phiprime}
V^\phi(t,x)&=\mathbb{E}\bigg[e^{-\beta(T-t)}g(X_T^{\phi + \varepsilon h}) \nonumber
\\
&\qquad -\int_t^T e^{-\beta(s-t)}\mathcal{L}^{\mu_{\phi+\varepsilon h}(s,X_s^{\phi + \varepsilon h})}[V^\phi](s,X_s^{\phi + \varepsilon h})\mathrm{d}s \bigg|X_t^{\phi + \varepsilon h}=x\bigg].
\end{align}
On the other hand, by the definition of \eqref{eq:cube-value}, we have
\begin{align}\label{eq:Vphiprime-def}
V^{\phi + \varepsilon h}(t,x)&=\mathbb{E}\bigg[e^{-\beta(T-t)}g(X_T^{\phi + \varepsilon h}) \nonumber
\\
&\qquad +\int_t^T e^{-\beta(s-t)} r\left(s,X_s^{\phi + \varepsilon h},\mu_{\phi+ \varepsilon h}(s,X_s^{\phi + \varepsilon h})\right)\mathrm{d} s \bigg| X_t^{\phi + \varepsilon h}=x\bigg].
\end{align}
Subtracting \eqref{eq:val-under-phiprime} from \eqref{eq:Vphiprime-def} gives that
\begin{align*}
V^{\phi+\varepsilon h}(t,x)-V^\phi(t,x)&=\mathbb{E}\Bigg[\int_t^T e^{-\beta(s-t)}\big(r(s,X_s^{\phi + \varepsilon h},\mu_{\phi+\varepsilon h}(s,X_s^{\phi + \varepsilon h})) \nonumber\\
&\qquad\qquad +\mathcal{L}^{\mu_{\phi+\varepsilon h}(s,X_s^{\phi + \varepsilon h})}[V^\phi](s,X_s^{\phi + \varepsilon h})\big)\mathrm{d}s\Big|X_t^{\phi + \varepsilon h}=x\Bigg]\\
&=\mathbb{E}\left[\int_t^T e^{-\beta(s-t)}q_\phi(s,X_s^{\phi + \varepsilon h},\mu_{\phi+\varepsilon h}(s,X_s^{\phi + \varepsilon h}))\mathrm{d}s\Big|X_t^{\phi + \varepsilon h}=x\right].
\end{align*}
By using Lemma \ref{lem:boundary-zero-time} and \eqref{eq:adv-zero}, we obtain
\begin{align*}
V^{\phi+\varepsilon h}(t,x)-V^\phi(t,x)&=\mathbb{E}\Bigg[\int_t^T e^{-\beta(s-t)}\Big(q_\phi(s,X_s^{\phi + \varepsilon h},\mu_{\phi+\varepsilon h}(s,X_s^{\phi + \varepsilon h}))\\
&\qquad \qquad -q_\phi(s,X_s^{\phi + \varepsilon h},\mu_{\phi}(s,X_s^{\phi + \varepsilon h}))\Big)\mathrm{d} s\Big|X_t^{\phi + \varepsilon h}=x\Bigg],
\end{align*}
which yields that
\begin{align*}
&\frac{V^{\phi+ \varepsilon h}(t,x)-V^\phi(t,x)}{\varepsilon}   
\\
&=\mathbb{E}\Bigg[\int_t^T e^{-\beta(s-t)}\frac{q_\phi\big(s,X_s^{\phi+\varepsilon h},\mu_{\phi+\varepsilon h}(s,X_s^{\phi+\varepsilon h})\big)-q_\phi\big(s,X_s^{\phi+\varepsilon h},\mu_{\phi}(s,X_s^{\phi+\varepsilon h})\big)}{\varepsilon}\mathrm{d} s \bigg|X_t^{\phi + \varepsilon h}=x\Bigg].
\end{align*}
This yields that
\begin{align}\label{eq:q-phi-diff}
&\frac{q_\phi\big(s,X_s^{\phi+\varepsilon h},\mu_{\phi+\varepsilon h}(s,X_s^{\phi+\varepsilon h})\big)-q_\phi\big(s,X_s^{\phi+\varepsilon h},\mu_{\phi}(s,X_s^{\phi+\varepsilon h})\big)}{\varepsilon} \nonumber
\\
&\qquad=\int_0^1\nabla_a q_\phi(s,X_s^{\phi+\varepsilon h},\mu_{\phi+\theta\varepsilon h}(s,X_s^{\phi+\varepsilon h}))^\top \nabla_\phi\mu_{\phi+\theta\varepsilon h}(s,X_s^{\phi+\varepsilon h})h\mathrm{d}\theta.
\end{align}
By Lemma \ref{lem:cube-stability}, $X_s^{\phi+\varepsilon h} \to X_s^{\phi}$ in probability for every $s \in [t,T]$. Moreover, the continuity of $(\phi,s,x)\mapsto\mu_\phi(s,x)$ and $(\phi,s,x)\mapsto\nabla_\phi\mu_\phi(s,x)$, together with the continuity of $\nabla_aq_\phi$ in the state and action variables, implies that the integrand on the right side of \eqref{eq:q-phi-diff} converges to $\nabla_a q_\phi \left(s,X_s^\phi,\mu_\phi(s,X_s^\phi) \right)^\top\nabla_\phi\mu_\phi(s,X_s^\phi)h$ in probability, $\mathrm{d}s\mathrm{d}\theta$-a.e..By Assumption \ref{ass:value} and \eqref{eq:nabla-q-funtion}, there exists a deterministic function $\Psi\in L^1([t,T])$ independent of $\varepsilon$ such that $|\nabla_a q_\phi(s,x,a)|\le\Psi(s)$ for all $(s,x,a)\in [t,T) \times Q\times A$. Furthermore, for sufficiently small $\varepsilon$, $|\nabla_\phi\mu_{\phi+\theta\varepsilon h}|$ is uniformly bounded by a deterministic constant because $X_s^{\phi+\varepsilon h}\in Q$ and $(\phi,s,x)\mapsto \nabla_\phi\mu_\phi(s,x)$ is continuous on a
compact set. Hence, we have from DCT that
\begin{align*}
\lim_{\varepsilon\to0}\frac{V^{\phi+\varepsilon h}(t,x)-V^\phi(t,x)}{\varepsilon}=\mathbb{E}\Bigg[\int_t^T e^{-\beta(s-t)}\nabla_a q_\phi\big(s,X_s^{\phi},\mu_\phi(s,X_s^{\phi})\big)^\top\nabla_\phi\mu_\phi(s,X_s^{\phi})h\mathrm{d}s \bigg|X_t^{\phi}=x\Bigg],    
\end{align*}
which completes the proof and yields \eqref{eq:cube-dpg}.
\end{proof}

Proposition \ref{prop:cube-dpg} extends the classical deterministic policy gradient of \cite{cheng25} to the continuous-time setting with reflected state dynamics. Notably, the reflection terms are absent from the final gradient formula: the Neumann boundary condition exactly cancels the boundary regulator contributions, ensuring that the actor update depends only on the interior dynamics.

\subsection{Martingale characterization on the hypercube}
\label{sec:martingale}

We now give the martingale characterization used for learning the value function and the advantage-rate function. The conditions imposed on the candidate pair $(\widehat{V},\widehat{q})$ are consistent with the Neumann Bellman equation and the properties of the advantage-rate function. A key feature of the following characterization is that the martingale condition is not imposed only along the closed-loop trajectory generated by $\mu_\phi$. Instead, it is required to hold under the locally perturbed of $\mu_\phi(t,x)$. This leads to an off-policy learning scheme with exploration.

\begin{theorem}\label{thm:cube-martingale}
Let assumptions of Lemma \ref{lem:verification} hold. Fix $\phi\in\mathbb R^k$. Let the candidate value function $\widehat{V} \in C^{1,2}([0,T)\times Q^{\circ})\cap C([0,T]\times Q)$ with $\nabla_x \widehat{V} \in C([0,T) \times Q)$, and the candidate advantage function $\widehat{q} \in C([0,T)\times Q\times A)$ satisfy
\begin{align}
\widehat{V}(T,x)=g(x),~\forall x\in Q,&\quad\partial_{x_i}\widehat{V}(t,x)=0,~\forall(t,x)\in[0,T)\times\partial Q~\text{and}~i\in I_-(x)\cup I_+(x),\label{eq:cube-hatV-boundary-condition}\\
&\widehat{q}(t,x,\mu_\phi(t,x))=0,\quad\forall (t,x)\in[0,T)\times Q^\circ.\label{eq:cube-hatq-zero}
\end{align}
Assume moreover that for every $(t,x)\in[0,T)\times Q^\circ$, there exists
a neighborhood $\mathcal{O}_{t,x}\subset A$ of $\mu_\phi(t,x)$ such that,
for every $a\in\mathcal{O}_{t,x}$, there exists a square-integrable $A$-valued adapted control process $\alpha=(\alpha_s)_{s\in[t,T]}$ with $\lim_{s\downarrow t}\alpha_s=a$, a.s.,
under which the process
\begin{align*}
\Bigg(e^{-\beta(s-t)}\widehat{V}(s,X_s^{t,x,\alpha})+\int_t^s e^{-\beta(u-t)}\left(r(u,X_u^{t,x,\alpha},\alpha_u)-\widehat{q}(u,X_u^{t,x,\alpha},\alpha_u)\right)\mathrm{d} u\Bigg)_{s\in[t,T]}
\end{align*}
is an $\mathbb{F}$-martingale, where the reflected state process $X^{t,x,\alpha}$ solves the dynamics:
\begin{align*}
\mathrm{d}X_s^{t,x,\alpha}=b(s,X_s^{t,x,\alpha},\alpha_s)\mathrm{d}s+\sigma(s,X_s^{t,x,\alpha},\alpha_s)\mathrm{d}W_s+\mathrm{d}K_s^- - \mathrm{d}K_s^+, \quad X_t^{t,x,\alpha}=x.
\end{align*}
Then, $\widehat{V}(t,x)=V^\phi(t,x)$ for $(t,x)\in[0,T]\times Q$ and $\widehat{q}(t,x,a)=q_\phi(t,x,a)$ for $(t,x,a)\in[0,T)\times Q^\circ \times \mathcal{O}_{t,x}$.
\end{theorem}

\begin{proof}
Define the function $F(t,x,a):=\mathcal{L}^a[\widehat{V}](t,x)+r(t,x,a)-\widehat{q}(t,x,a)$ for $(t,x,a)\in[0,T)\times Q^\circ \times A$. We prove that $F(t,x,a)=0$ for any $(t,x,a)\in[0,T)\times Q^\circ \times \mathcal{O}_{t,x}$. Suppose that there exist $(t,x)\in[0,T)\times Q^\circ$ and $a\in\mathcal{O}_{t,x}$ such that $F(t,x,a)>0$. By using the continuity of $(t,x,a)\mapsto F(t,x,a)$, there exist $\eta,\delta>0$ such that $F(\bar{t},\bar{x},\bar{a})\ge \eta$ for $(\bar{t},\bar{x},\bar{a})$ satisfying $\max\{|\bar{t}-t|,\ |\bar{x}-x|,\ |\bar{a}-a|\}<\delta$. Without loss of generality, we may take $\delta<T-t$. Then there exists an $A$-valued adapted control process $\alpha=(\alpha_s)_{s\in[t,T]}$ with $\lim_{s\downarrow t}\alpha_s=a$, a.s. such that the process defined by
\begin{align*}
M_s:=e^{-\beta(s-t)}\widehat{V}(s,X_s^{t,x,\alpha})+\int_t^s e^{-\beta(u-t)}\left(r(u,X_u^{t,x,\alpha},\alpha_u)-\widehat{q}(u,X_u^{t,x,\alpha},\alpha_u)\right)\mathrm{d}u,\quad \forall s\in[t,T]   
\end{align*}
is an $\mathbb{F}$-martingale. Let the stopping time $\tau:=\inf\{s\ge t: \max\{|s-t|,|X_s^{t,x,\alpha}-x|,|\alpha_s-a|\}\ge \delta \} \wedge (t+\delta)$. By the continuity of $s\to X_s^{t,x,\alpha}$ and the property $\alpha_s\to a$ as $s\downarrow t$, we have $\tau>t$ a.s.. For $\varepsilon\in(0,T-t)$, applying the It\^o rule to
$e^{-\beta(s-t)}\widehat{V}(s,X_s^{t,x,\alpha})$ on $[t,T-\varepsilon]$, we obtain
\begin{align*}
\mathrm{d}M_s&=e^{-\beta(s-t)}(\mathcal{L}^{\alpha_s}[\widehat{V}](s,X_s^{t,x,\alpha})+r(s,X_s^{t,x,\alpha},\alpha_s)-\widehat{q}(s,X_s^{t,x,\alpha},\alpha_s))\mathrm{d}s
\\
&\quad+e^{-\beta(s-t)}\nabla_x\widehat{V}(s,X_s^{t,x,\alpha})^\top \sigma(s,X_s^{t,x,\alpha},\alpha_s)\mathrm{d} W_s
\\
&\quad+e^{-\beta(s-t)}
\sum_{i=1}^n \partial_{x_i}\widehat{V}(s,X_s^{t,x,\alpha})\mathrm{d} K_s^{-,i}-e^{-\beta(s-t)}\sum_{i=1}^n \partial_{x_i}\widehat{V}(s,X_s^{t,x,\alpha})\mathrm{d}K_s^{+,i}.
\end{align*}
By the boundary condition \eqref{eq:cube-hatV-boundary-condition}, the reflecting terms vanish. Since $(M_s)_{s\in[t,T-\varepsilon]}$ is still a martingale, its finite-variation part must vanish. Therefore, for all $s\in[t,T-\varepsilon]$,
\begin{align}\label{eq:cube-finite-variation-zero}
\int_t^s e^{-\beta(u-t)}\left(\mathcal{L}^{\alpha_u}[\widehat{V}](u,X_u^{t,x,\alpha})+r(u,X_u^{t,x,\alpha},\alpha_u)-\widehat{q}(u,X_u^{t,x,\alpha},\alpha_u)\right)\mathrm{d}u=0,\quad \text{a.s.}  
\end{align}
Since $\tau\le t+\delta<T$, we may apply
\eqref{eq:cube-finite-variation-zero} up to time $\tau$, and obtain that
\begin{align*}
0=\int_{t}^{\tau} e^{-\beta(u-t)}F(u,X_u^{t,x,\alpha},\alpha_u)\mathrm{d}u\ge\eta\int_{t}^{\tau} e^{-\beta(u-t)}\mathrm{d}u>0,\quad \text{a.s.},    
\end{align*}
which is a contradiction. Hence, $\widehat{q}(t,x,a)=\mathcal{L}^a[\widehat{V}](t,x)+r(t,x,a)$ for all $(t,x,a)\in[0,T)\times Q^\circ \times \mathcal{O}_{t,x}$. Taking $a=\mu_\phi(t,x) \in \mathcal{O}_{t,x}$ and using
\eqref{eq:cube-hatq-zero}, we get $\mathcal{L}^{\mu_\phi(t,x)}[\widehat{V}](t,x)+r(t,x,\mu_\phi(t,x))=0$ for all $(t,x)\in[0,T)\times Q^\circ$. Together with the boundary condition \eqref{eq:cube-hatV-boundary-condition}, this shows that
$\widehat{V}$ is a classical solution of the Neumann boundary problem
\eqref{eq:neumann-pde}. By Lemma \ref{lem:verification}, it follows that $\widehat{V}=V^\phi$ on $[0,T]\times Q$. Then, we conclude that $\widehat{q}(t,x,a)=\mathcal{L}^a[V^\phi](t,x)+r(t,x,a)=q_\phi(t,x,a)$ for all $(t,x)\in[0,T)\times Q^\circ$ and $a\in\mathcal O_{t,x}$.
\end{proof}

\section{Model-free CT-DDPG-R Algorithm}
\label{sec:algorithm}

In this section, we design the model-free continuous-time actor-critic algorithm for reflected diffusion control based on the theoretical results in Section \ref{sec:dpg-m}. To ensure that the parameterized neural networks are consistent with the conditions \eqref{eq:cube-hatV-boundary-condition} required by the theoretical framework, we introduce two boundary-aware approximation methods: a soft penalty formulation and a hard architectural constraint. We also analyze the discrepancy between the ideal continuous-time state process under the target policy and the state process generated in practical implementation by discretely sampled, piecewise-constant exploratory controls, and derive an upper bound for this approximation error.

\subsection{Algorithm design}
\label{sec:algorithm-design}
Unlike model-based control approaches, our continuous-time RL framework does not assume prior knowledge of the coefficients $b$ and $\sigma$ governing the state dynamics and the reward functions $r$ and $g$. Instead, we assume only that the agent can interact with the environment and collect trajectory data in the form of state-action-reward tuples.  Proposition \ref{prop:cube-dpg} and Theorem \ref{thm:cube-martingale} provide the theoretical underpinnings for a model-free actor-critic implementation: the martingale characterization enables critic learning from temporal-difference (TD) residuals, while the policy gradient theorem guides the actor update toward improved performance.

We introduce parametric function approximators $V_\theta:[0,T]\times Q\to\mathbb R,~\bar q_\psi:[0,T]\times Q\times A\to\mathbb R,~\mu_\phi:[0,T]\times Q\to A$ for the value function, the advantage-rate function, and the deterministic policy,
where $\theta$, $\psi$, and $\phi$ denote the respective parameter vectors. To avoid that the parameters $\theta,\phi$ for value network and policy network change too fast, we use target networks $\bar{\theta},\bar{\phi}$ updated by
\begin{align*}
\bar{\theta} \leftarrow \tau \theta+(1-\tau)\bar{\theta},\qquad \bar{\phi} \leftarrow \tau \phi + (1-\tau) \bar{\phi}.
\end{align*}
with $\tau\in(0,1)$ controlling the update rate. Following \cite{cheng25}, we enforce the normalization condition \eqref{eq:cube-hatq-zero} via the reparameterization
\begin{align*}
q_\psi(t,x,a):=\bar{q}_\psi(t,x,a)-\bar{q}_\psi\left(t,x,\mu_{\bar{\phi}}(t,x)\right),
\end{align*}
for the current policy $\mu_\phi$ when training the critic.

In our paper, two methods are adopted to address the Neumann boundary condition:
\begin{itemize}
\item \textbf{Soft constraint.}\quad The soft constraint approach uses a standard neural network $V_\theta(t,x)$ without architectural modifications, but adds the critic loss with a boundary penalty term. Specifically, we sample boundary points $\{(t^{(j)},x^{(j)})\}_{j=1}^{B_N}$ from the faces of the hypercube and penalize the squared normal derivatives of the value function:
\begin{align*}
\mathcal{L}_N(\theta)=\frac{1}{B_N}\sum_{j=1}^{B_N}\sum_{i\in  I_-(x^{(j)})\cup  I_+(x^{(j)})}\left|\partial_{x_i}V_\theta(t^{(j)},x^{(j)})\right|^2.
\end{align*}

\item \textbf{Hard constraint.}\quad For the hypercube domain, the Neumann boundary condition can be imposed exactly by constructing the value network using cosine Fourier features. Inspired by \cite{str25}, we embed the state coordinates through cosine functions of multiple frequencies, which naturally satisfy zero normal derivatives at the boundaries. The value network takes the form
\begin{align}\label{eq:hard-constraint-value}
V_{\theta}(t,x)=\bar{V}_{\theta}\bigg(t,&\cos\left(\pi  \frac{x_1-\underline{x}_1}{\overline{x}_1-\underline{x}_1}\right),\cos\left(\pi  \frac{x_2-\underline{x}_2}{\overline{x}_2-\underline{x}_2}\right),\cdots,\cos\left(\pi  \frac{x_d-\underline{x}_d}{\overline{x}_d-\underline{x}_d}\right),\nonumber
\\
&\cos\left(f_1\pi \frac{x_1-\underline{x}_1}{\overline{x}_1-\underline{x}_1}\right),\cos\left(f_1\pi  \frac{x_2-\underline{x}_2}{\overline{x}_2-\underline{x}_2}\right),\cdots,\cos\left(f_1\pi  \frac{x_d-\underline{x}_d}{\overline{x}_d-\underline{x}_d}\right),
\\
&\cos\left(f_M\pi\frac{x_1-\underline{x}_1}{\overline{x}_1-\underline{x}_1}\right),\cos\left(f_M\pi  \frac{x_2-\underline{x}_2}{\overline{x}_2-\underline{x}_2}\right),\cdots,\cos\left(f_M\pi  \frac{x_d-\underline{x}_d}{\overline{x}_d-\underline{x}_d}\right)\bigg) \nonumber,
\end{align}
where the integer frequencies $f_1,f_2,\cdots,f_M \in \mathbb{Z}_+$ are chosen in order to approximate the high frequency or multiscale data.
It is obvious that the neural network \eqref{eq:hard-constraint-value} incorporates the Neumann condition by the property of cosine function. 
\end{itemize}

Fix a time step $\Delta t$ and we get the time grid $0=t_0<t_1<\cdots<t_N=T$ with $\Delta t_k:=t_{k+1}-t_k=\Delta t$. In step $k$, the agent observes the state $x_k\in Q$, selects an exploration action $a_k\in A$,
and receives the next state $x_{k+1}\in Q$ together with the running reward $r_k$.
We store transitions of the form $(t_k,x_k,a_k,r_k)$
in a replay buffer $\mathcal{R}$. To encourage sufficient exploration of the state-action space, we adopt a perturbed version of the current deterministic policy:
\begin{align*}
a_k=\Pi_A\left(\mu_\phi(t_k,x_k)+\eta_k\right),\quad \eta_k\sim \mathcal N(0,\sigma_{\mathrm{exp}}^2 I_d),
\end{align*}
where $\sigma_{\mathrm{exp}}^2$ is a small variance to explore the behavior around $\mu_{\phi}(t_k,x_k)$, $I_d$ is the 
$d$-dimensional identity matrix, and $\Pi_A$ denotes projection onto the action set $A$ when needed. This exploration scheme injects Gaussian noise into the action, allowing the agent to explore the environment around the current policy and gather informative samples for learning the advantage-rate function.

We adopt the multi-step TD method when training the value network and the advantage-rate network. For any $L$-step sampled subsequence $\{(t_{k+\ell},x_{k+\ell},a_{k+\ell},r_{k+\ell})\}_{\ell=0}^L$ which is sampled from the replay buffer $\mathcal{R}$, we define the following martingale TD given by
\begin{align*}
\delta_{k,L}^{\theta,\psi,\bar\theta}&:=
V_\theta(t_k,x_k)-\sum_{\ell=0}^{L-1}e^{-\beta(t_{k+\ell}-t_k)}\left(
r_{k+\ell}-q_\psi(t_{k+\ell},x_{k+\ell},a_{k+\ell})\right)\Delta t
-e^{-\beta(t_{k+L}-t_k)}V_{\bar\theta}(t_{k+L},x_{k+L}),
\end{align*}
where $V_{\bar\theta}$ is the target value network. This residual is precisely the discrete-time approximation of the martingale condition from Theorem \ref{thm:cube-martingale}. After sampling a batch of $B$ such subsequences, we compute the martingale loss
\begin{align*}
\mathcal{L}_M(\theta,\psi):=\frac{1}{B}\sum_{j=1}^B\left|\delta_{k_j,L}^{\theta,\psi,\bar\theta}\right|^2.
\end{align*}
Minimizing this loss drives the critic toward satisfying the martingale characterization, thereby aligning the value and advantage-rate estimates with the underlying dynamics. To enforce the terminal condition $V^\phi(T,x)=g(x)$, we also add a terminal penalty
\begin{align*}
\mathcal{L}_T(\theta):=\frac1{B_T}\sum_{j=1}^{B_T}\left|V_\theta(T,x^{(j)})-g(x^{(j)})\right|^2,
\end{align*}
where the terminal states $\{x^{(j)}\}_{j=1}^{B_T}$ are sampled randomly from $Q$.

The overall critic loss is therefore a weighted combination of the martingale loss, the terminal penalty and, if using the soft constraint, the boundary penalty:
\begin{align*}
\mathcal{L}_{\mathrm{critic}}(\theta,\psi)=\mathcal{L}_M(\theta,\psi)+\lambda_T \mathcal L_T(\theta) + \lambda_N \mathcal{L}_N(\theta),
\end{align*}
with weights $\lambda_T>0$ and $\lambda_N>0$ (where $\lambda_N=0$ for the hard constraint). The critic parameters are then updated via semi-gradient descent:
\begin{align*}
\theta \leftarrow \theta - \alpha_\theta \partial_\theta \mathcal{L}_{\mathrm{critic}},
\qquad
\psi \leftarrow \psi - \alpha_\psi \partial_\psi \mathcal{L}_{\mathrm{critic}}.   
\end{align*}

Motivated by the deterministic policy gradient in Proposition \ref{prop:cube-dpg}, we update the actor parameters by ascending the following objective:
\begin{align*}
\mathcal{J}_{\mathrm{actor}}(\phi):=\frac{1}{B_A}\sum_{j=1}^{B_A}\bar{q}_\psi\left(t^{(j)},x^{(j)},\mu_\phi(t^{(j)},x^{(j)})\right) \Delta t,
\end{align*}
where $\{(t^{(j)},x^{(j)})\}_{j=1}^{B_A}$ are sampled from replay buffer $\mathcal{R}$ and the actor parameters are updated as
\begin{align*}
\phi \leftarrow \phi + \alpha_\phi\partial_\phi \mathcal{J}_{\mathrm{actor}}.    
\end{align*}
The full procedure of continuous-time deep deterministic policy gradient with reflection (CT-DDPG-R) on hypercube is described in Algorithm \ref{alg:ct-ddpg-r}.

\begin{algorithm}[ht]
\caption{CT-DDPG-R on the hypercube}
\label{alg:ct-ddpg-r}
\begin{algorithmic}[1]
\Require
Time step $\Delta t$, time grid $0=t_0<t_1<\cdots<t_K=T$, discount rate $\beta$, number of episodes $N$, policy net $\mu_\phi$, advantage-rate net $\bar{q}_\psi$, value net $V_\theta$, critic update frequency $f_{\mathrm{critic}}$, actor update frequency $f_{\mathrm{actor}}$, trajectory length $L$, exploration noise $\sigma_{\mathrm{exp}}^2$, soft update parameter $\tau$, learning rates $\alpha_\phi,\alpha_\psi,\alpha_\theta$, batch sizes $B,B_A,B_T,B_N$, loss weight $\lambda_T,\lambda_N$.
\State Initialize $\phi,\psi,\theta$, target $\bar{\theta}=\theta,\bar{\phi}=\phi$ and buffer $\mathcal{R}$
\For{$n=1,2,\dots,N$}
    \State Observe the initial state $x_0$, collect the transitions from a trajectory by exploratory action
    \begin{align*}
    a_{t_k}=\Pi_A\left(\mu_\phi(t_k,x_{t_k})+\eta_{t_k}\right),~\eta_k\sim\mathcal{N}(0,\sigma_{\mathrm{exp}}^2 I_d).    
    \end{align*}                 
    \For{$m=1,\ldots,f_{\mathrm{critic}}$}
        \State Sample a batch of $L$-step trajectories and compute the martingale loss $\mathcal{L}_M(\theta,\psi)$.
    
        \State Sample a batch of terminal points and compute the terminal penalty $\mathcal{L}_T(\theta)$.

        \State Sample a batch of boundary points and compute the boundary penalty $\mathcal{L}_N(\theta)$.

        \State Compute the critic loss $\mathcal{L}_{\mathrm{critic}}(\theta,\psi)=\mathcal{L}_M(\theta,\psi)+\lambda_T\mathcal{L}_T(\theta)+\lambda_N \mathcal{L}_N(\theta)$.

        \State Update $(\theta,\psi)$ by $\theta \leftarrow \theta - \alpha_\theta \partial_\theta \mathcal{L}_{\mathrm{critic}}$ and $\psi \leftarrow \psi - \alpha_\psi \partial_\psi\mathcal{L}_{\mathrm{critic}}$.

    \EndFor
    \State Sample a batch of states from $\mathcal{R}$ and compute the policy loss $\mathcal{J}_{\mathrm{actor}}(\phi)$.
            
    \State Update $\phi$ by $\phi \leftarrow \phi + \alpha_\phi \partial_\phi \mathcal{J}_{\mathrm{actor}}$.
    
    \State Update target networks: $\bar{\theta} \leftarrow \tau\theta+(1-\tau)\bar{\theta},~ \bar{\phi} \leftarrow \tau\phi+(1-\tau)\bar{\phi}$.
\EndFor
\end{algorithmic}
\end{algorithm}

\subsection{Accuracy of discretely sampled exploratory dynamics}
\label{sec:accuracy-exploration}

We now turn to quantify the discrepancy between the ideal continuous-time reflected dynamics governed by the deterministic policy $\mu_\phi$ and the dynamics actually executed during training. This analysis is crucial for understanding the approximation error introduced by the discrete-time implementation of our otherwise continuous-time algorithm. Let
\begin{align*}
\mathcal{G}:=\{(t_0,t_1,...,t_K):~0=t_0<t_1<\cdots<t_K=T\},\qquad |\mathcal{G}|:=\max_{0\le k\le K-1}(t_{k+1}-t_k)    
\end{align*}
denote a uniform time grid and its mesh size, respectively. For $t\in[t_k,t_{k+1})$, define the left-endpoint projection $\iota(t):=t_k$. To model the exploration noise injected by the algorithm, let $\{\zeta_k\}_{k\ge0}$ be i.i.d. random variables in $\mathbb{R}^d$
satisfying $\mathbb{E}[\zeta_k]=0, \mathbb{E}|\zeta_k|^2=M_\zeta<\infty$, and let $\sigma_{\mathrm{exp}}>0$ be the exploration scale.

\subsubsection{Extension of the probability space}
To accommodate the independent random exploration noise $\{\zeta_k\}_{k \geq 0}$ generated by the algorithm, we extend the underlying probability space accordingly. Let $(\Omega^\zeta,\mathcal{F}^\zeta,\mathbb{P}^\zeta)$ be a probability space supporting $(\zeta_k)_{k\ge0}$. Define the product space
$\bar{\Omega}:=\Omega\times\Omega^\zeta$, $\bar{\mathcal{F}}:=\mathcal{F}\otimes\mathcal{F}^\zeta$, $\bar{\mathbb{P}}:=\mathbb{P}\otimes\mathbb{P}^\zeta$. On the extended space $(\bar{\Omega},\bar{\mathcal{F}},\bar{\mathbb{P}})$, we set $\bar{W}_t(\omega,\omega^\zeta):=W_t(\omega)$, $\bar{\xi}_0(\omega,\omega^\zeta):=\xi_0(\omega)$, $\bar{\zeta}_k(\omega,\omega^\zeta):=\zeta_k(\omega^\zeta)$. For notational simplicity, we retain the original symbols $(W,\xi_0,\zeta_k)$ in place of $(\bar{W},\bar{\xi}_0,\bar{\zeta}_k)$. We equip the extended space with the enlarged filtration $\bar{\mathcal{F}}_t^{\mathcal{G}}:=\sigma(\xi_0,\,W_s:0\le s\le t)\vee\sigma(\zeta_k:t_k\le t)\vee\mathcal{N}$ for $t\in[0,T]$, where $\mathcal{N}$ denotes the family of $\bar{\mathbb{P}}$-null sets. By
independence, $W$ remains a Brownian motion with respect to $\{\bar{\mathcal{F}}_t^{\mathcal{G}}\}_{t\in[0,T]}$.

\subsubsection{Discretely executed exploratory reflected dynamics}

In the actual implementation, the agent updates its action only at the discrete grid points $t_k$. At each intervention time $t_k$, the algorithm samples the exploratory action
\begin{align*}
a_k^{\mathcal{G},\sigma_{\mathrm{exp}}}:= \Pi_A \left(\mu_\phi\left(t_k,X_{t_k}^{\mathcal{G},\sigma_{\mathrm{exp}}}\right)+\sigma_{\mathrm{ exp}}\zeta_k\right),   
\end{align*}
and holds it constant on $[t_k,t_{k+1})$, that is, $a_t^{\mathcal{G},\sigma_{\mathrm{ exp}}}=a_k^{\mathcal{G},\sigma_{\mathrm{exp}}}$ for $t\in[t_k,t_{k+1})$. This piecewise-constant action construction mirrors the actual execution of the algorithm, where the policy is evaluated and the action is updated at regular intervals of length $\Delta t$.

Consequently, the reflected state process under this discretely sampled exploratory policy satisfies
\begin{align}\label{eq:sampled-exploratory-rsde}
\begin{cases}
\begin{aligned}
\displaystyle 
\mathrm{d}X_t^{\mathcal{G},\sigma_{\mathrm{exp}}}&=b\left(t,X_t^{\mathcal{G},\sigma_{\mathrm{exp}}},a_t^{\mathcal{G},\sigma_{\mathrm{exp}}}\right)\mathrm{d}t +\sigma\left(t,X_t^{\mathcal{G},\sigma_{\mathrm{exp}}},a_t^{\mathcal{G},\sigma_{\mathrm{exp}}}\right)\mathrm{d}W_t
\\
&\qquad + \mathrm{d}K_t^{\mathcal{G},\sigma_{\mathrm{exp}},-}-\mathrm{d}K_t^{\mathcal{G},\sigma_{\mathrm{exp}},+},
\end{aligned}
\\
\displaystyle
X_0^{\mathcal{G},\sigma_{\mathrm{exp}}}=\xi_0.
\end{cases}
\end{align}
where the regulators $K^{\mathcal{G},\sigma_{\mathrm{exp}},\pm}$ enforce the reflection on the boundary of $Q$. In contrast, the ideal deterministic reflected dynamics, which would be obtained if the action could be updated continuously in time, is governed by
\begin{align*}
\begin{cases}
\mathrm{d}X_t^\phi=b\left(t,X_t^\phi,\mu_\phi(t,X_t^\phi)\right)\mathrm{d}t+ \sigma\left(t,X_t^\phi,\mu_\phi(t,X_t^\phi)\right)\mathrm{d}W_t+ \mathrm{d}K_t^{\phi,-}-\mathrm{d}K_t^{\phi,+},
\\
X_0^\phi=\xi_0.
\end{cases}
\end{align*}
Both the sampled and the ideal processes are defined on the same product space $(\bar\Omega,\bar{\mathcal F},\bar{\mathbb P})$. The following lemma ensures that the exploratory algorithm is well-defined at every step of the simulation, justifying the use of the piecewise-constant action approximation.

\begin{lemma}[Well-posedness of the exploratory dynamics]
\label{prop:sampled-rsde-wellposed}
Under the assumptions of Lemma \ref{lem:verification}, for every grid $\mathcal{G}$ and every $\sigma_{\mathrm{exp}} > 0$, the reflected SDE \eqref{eq:sampled-exploratory-rsde} admits a unique strong solution $\left(X^{\mathcal{G},\sigma_{\mathrm{exp}}},K^{\mathcal{G},\sigma_{\mathrm{exp}},-},K^{\mathcal{G},\sigma_{\mathrm{exp}},+}\right)$ adapted to $\{\bar{\mathcal{F}}_t^{\mathcal{G}}\}_{t\in[0,T]}$.
\end{lemma}

\begin{proof}
We proceed by induction over the grid intervals. The case at $t=0$ is immediate from the initial condition $X_0^{\mathcal{G},\sigma_{\mathrm{exp}}}=\xi_0\in\bar{\mathcal{F}}_0^{\mathcal{G}}$. Suppose that for some $k\in\mathbb{N}$, the reflected SDE \eqref{eq:sampled-exploratory-rsde} admits a unique strong solution on $[0,t_k]$.
Then we have $X_{t_k}^{\mathcal{G},\sigma_{\mathrm{exp}}}\in\bar{\mathcal{F}}_{t_k}^{\mathcal{G}}$, and
\begin{align*}
a_k^{\mathcal{G},\sigma_{\mathrm{exp}}}=\Pi_A\left(\mu_\phi(t_k,X_{t_k}^{\mathcal{G},\sigma_{\mathrm{exp}}})+\sigma_{\mathrm{exp}}\zeta_k\right)\in\bar{\mathcal{F}}_{t_k}^{\mathcal{G}}.    
\end{align*}
On $[t_k,t_{k+1})$, we consider the following reflected SDE
\begin{align*}
\begin{cases}
\mathrm{d}X_t=b(t,X_t,a_k^{\mathcal{G},\sigma_{\mathrm{exp}}})\mathrm{d}t + \sigma(t,X_t,a_k^{\mathcal{G},\sigma_{\mathrm{exp}}})\mathrm{d}W_t +
\mathrm{d}K_t^- - \mathrm{d}K_t^+,
\\
X_{t_k}=X_{t_k}^{\mathcal{G},\sigma_{\mathrm{exp}}}.
\end{cases}
\end{align*}
For fixed $a_k^{\mathcal{G},\sigma_{\mathrm{exp}}}$, it follows from Assumptions \ref{ass:coefficient} and \ref{ass:policy} that
\begin{align*}
|b(t,x,a_k^{\mathcal{G},\sigma_{\mathrm{exp}}})-b(t,y,a_k^{\mathcal{G},\sigma_{\mathrm{exp}}})|+|\sigma(t,x,a_k^{\mathcal{G},\sigma_{\mathrm{exp}}})-\sigma(t,y,a_k^{\mathcal{G},\sigma_{\mathrm{exp}}})|\le C|x-y|,
\end{align*}
for some constant $C>0$. Therefore, the reflected SDE admits a unique strong solution on $[t_k,t_{k+1})$.  This extends the solution to $[0,t_{k+1}]$, completing the induction. Hence the unique strong solution exists on the entire horizon $[0,T]$.
\end{proof}

\subsubsection{State process error}
We now analyze the discrepancy between the ideal continuous-time trajectory and the discretely sampled exploratory trajectory. This analysis is essential for quantifying the approximation error introduced by the piecewise-constant action implementation and the exploration noise.

For notational convenience, we introduce the shorthand $a_t^\phi:=\mu_\phi(t,X_t^\phi),~K_t^{\mathcal{G},\sigma_{\mathrm{exp}}}:=K_t^{\mathcal{G},\sigma_{\mathrm{exp}},-}-K_t^{\mathcal{G},\sigma_{\mathrm{exp}},+},~K_t^\phi:=K_t^{\phi,-}-K_t^{\phi,+}$ and define the state deviation process $\Delta_t:=X_t^{\mathcal{G},\sigma_{\mathrm{exp}}}-X_t^\phi$. 

\begin{theorem}
\label{thm:state-process-error}
Under the assumptions of Lemma \ref{prop:sampled-rsde-wellposed}, there exists a constant $C>0$, independent of
$\mathcal{G}$ and $\sigma_{\mathrm{exp}}$, such that
\begin{align}\label{eq:state-process-error}
\mathbb{E}\left[\sup_{0\le t\le T}\left|X_t^{\mathcal{G},\sigma_{\mathrm{exp}}}-X_t^\phi\right|^2\right]\le C\left(|\mathcal{G}|+\sigma_{\mathrm{exp}}^2 M_\zeta\right).
\end{align}
\end{theorem}

\begin{proof}
Define the differences in drift and diffusion coefficients as
\begin{align*}
\Delta b_t:=b(t,X_t^{\mathcal{G},\sigma_{\mathrm{exp}}},a_t^{\mathcal{G},\sigma_{\mathrm{exp}}})-b(t,X_t^\phi,a_t^\phi),\quad \Delta\sigma_t:=\sigma(t,X_t^{\mathcal{G},\sigma_{\mathrm{exp}}},a_t^{\mathcal{G},\sigma_{\mathrm{exp}}})-\sigma(t,X_t^\phi,a_t^\phi).
\end{align*}
Then the deviation process satisfies the SDE
\begin{align*}
\mathrm{d}\Delta_t=\Delta b_t\mathrm{d}t+\Delta\sigma_t\mathrm{d}W_t+
\mathrm{d}K_t^{\mathcal{G},\sigma_{\mathrm{exp}}}-\mathrm{d}K_t^\phi.
\end{align*}
Applying  It\^o's formula to $|\Delta_t|^2$ gives that
\begin{align*}
|\Delta_t|^2=2\int_0^t \langle \Delta_s,\Delta b_s\rangle \mathrm{d}s + \int_0^t |\Delta\sigma_s|^2\mathrm{d}s+ 2\int_0^t
\langle \Delta_s,\Delta\sigma_s \mathrm{d} W_s\rangle+ 2\int_0^t\langle \Delta_s, \mathrm{d}K_s^{\mathcal{G},\sigma_{\mathrm{exp}}}-\mathrm{d}K_s^\phi \rangle.
\end{align*}

We claim that $\int_0^t \langle  \Delta_s,\mathrm{d} K^{\mathcal{G},\sigma_{\mathrm{exp}}}_s-\mathrm{d} K^{\phi}_s \rangle \leq 0$. For the lower face, if $\mathrm{d}K_s^{\mathcal{G},\sigma_{\mathrm{exp}},-,i}>0$, we have $X_s^{\mathcal{G},\sigma_{\mathrm{exp}},i}=\underline{x}_i$ and then, $X_s^{\mathcal{G},\sigma_{\mathrm{exp}},i}-X_s^{\phi,i}=\underline{x}_i-X_s^{\phi,i}\le 0$. For the upper face, if $\mathrm{d}K_s^{\mathcal{G},\sigma_{\mathrm{exp}},+,i}>0$, we have $X_s^{\mathcal{G},\sigma_{\mathrm{exp}},i}=\overline{x}_i$, and then $-(X_s^{\mathcal{G},\sigma_{\mathrm{exp}},i}-X_s^{\phi,i})=X_s^{\phi,i}-\overline{x}_i \le 0$. Therefore, we obtain that $\int_0^t \langle X_s^{\mathcal{G},\sigma_{\mathrm{exp}}}-X_s^\phi,\mathrm{d}K_s^{\mathcal{G},\sigma_{\mathrm{exp}}}\rangle \le 0$.
An analogous argument applied to the ideal regulator $K^\phi$ gives $-\int_0^t \langle X_s^{\mathcal{G},\sigma_{\mathrm{exp}}}-X_s^\phi,\mathrm{d}K_s^\phi\rangle \le 0$.
Adding the two inequalities gives $\int_0^t \langle\Delta_s,\mathrm{d}K_s^{\mathcal{G},\sigma_{\mathrm{exp}}}-\mathrm{d}K_s^\phi\rangle \le 0$, which yields
\begin{align}\label{eq:ineq-Delta-t}
|\Delta_t|^2 \le 2\int_0^t \langle \Delta_s,\Delta b_s\rangle \mathrm{d}s +\int_0^t|\Delta\sigma_s|^2\mathrm{d}s+2\int_0^t
\langle \Delta_s,\Delta\sigma_s \mathrm{d}W_s\rangle.
\end{align}
By the Lipschitz condition in Assumption \ref{ass:coefficient}, $|\Delta b_s|+|\Delta\sigma_s| \le C(|\Delta_s|+|a_s^{\mathcal{G},\sigma_{\mathrm{exp}}}-a_s^\phi|)$ for some constant $C>0$. Therefore, $2\langle\Delta_s,\Delta b_s\rangle+|\Delta\sigma_s|^2 \le C|\Delta_s|^2+
C|a_s^{\mathcal{G},\sigma_{\mathrm{exp}}}-a_s^\phi|^2$. Let $F(t):=\mathbb{E}[\sup_{0\le u\le t}|\Delta_u|^2]$.
Taking expectations and supremum on both sides of \eqref{eq:ineq-Delta-t}, and applying the BDG inequality and Young's inequality, we obtain
\begin{align*}
F(t)&\le C\int_0^t \mathbb{E}|\Delta_s|^2\mathrm{d}s+C\int_0^t
\mathbb{E}|a_s^{\mathcal{G},\sigma_{\mathrm{exp}}}-a_s^\phi|^2\mathrm{d}s+C\mathbb{E}\left[\left(\int_0^t|\Delta_s|^2|\Delta\sigma_s|^2\mathrm{d}s
\right)^{1/2}\right]
\\
&\le \frac{1}{2} F(t)+C\int_0^tF(s)\mathrm{d}s+C\int_0^t
\mathbb{E}|a_s^{\mathcal{G}\sigma_{\mathrm{exp}}}-a_s^\phi|^2\mathrm{d}s.
\end{align*}
Thus, it follows that
\begin{align}\label{eq:F-before-action}
F(t)\le C\int_0^t F(s)\mathrm{d}s+C\int_0^t\mathbb{E}|a_s^{\mathcal{G},\sigma_{\mathrm{exp}}}-a_s^\phi|^2\mathrm{d}s.
\end{align}

It remains to bound the action discrepancy. For $s\in[t_k,t_{k+1})$, since $\Pi_A$ is nonexpansive and $\Pi_A(a_s^\phi)=a_s^\phi$, we have
\begin{align*}
|a_s^{\mathcal{G},\sigma_{\mathrm{exp}}}-a_s^\phi|&= \left|\Pi_A\left(\mu_\phi(t_k,X_{t_k}^{\mathcal{G},\sigma_{\mathrm{exp}}})+\sigma_{\mathrm{exp}}\zeta_k\right)-\Pi_A\left(\mu_\phi(s,X_s^\phi)\right)\right|
\\
&\le\left|\mu_\phi(t_k,X_{t_k}^{\mathcal{G},\sigma_{\mathrm{exp}}})-\mu_\phi(s,X_s^\phi)+\sigma_{\mathrm{exp}}\zeta_k\right|
\\
&\le C\left(|X_{t_k}^{\mathcal{G},\sigma_{\mathrm{exp}}}-X_{t_k}^\phi|+|X_{t_k}^\phi-X_s^\phi|+|s-t_k|+\sigma_{\mathrm{exp}}|\zeta_k|\right)
\\
&\le C\left( \sup_{0\le u\le s}|\Delta_u| +|X_{\iota(s)}^\phi-X_s^\phi|+|\mathcal{G}|+\sigma_{\mathrm{exp}}|\zeta_k|\right).
\end{align*}
For the second term on the right side, we apply It\^o's formula and obtain
\begin{align*}
|X_{s}^\phi - X_{\iota(s)}^\phi|^2&= 2 \int_{\iota(s)}^s \langle X_u^\phi - X_{\iota(s)}^\phi, b(u, X_u^\phi, a_u^\phi) \rangle \mathrm{d}u + \int_{\iota(s)}^s \mathrm{Tr}\left(\sigma\sigma^\top(u, X_u^\phi, a_u^\phi)\right) \mathrm{d}u \nonumber
\\
& + 2 \int_{\iota(s)}^s \langle X_u^\phi - X_{\iota(s)}^\phi, \sigma(u, X_u^\phi, a_u^\phi) \mathrm{d} W_u \rangle + 2 \int_{\iota(s)}^s \langle X_u^\phi - X_{\iota(s)}^\phi, \mathrm{d}K_u^{\phi,-} - \mathrm{d} K_u^{\phi,+} \rangle.
\end{align*}
By the same argument as before, we know that the reflection term is non-positive, that is
\begin{align*}
\int_{\iota(s)}^s \langle X_u^\phi - X_{\iota(s)}^\phi, \mathrm{d}K_u^{\phi,-} - \mathrm{d} K_u^{\phi,+} \rangle \le 0.
\end{align*}
Then we obtain
\begin{align*}
\mathbb{E}|X_s^\phi - X_{\iota(s)}^\phi|^2 &\le \mathbb{E} \int_{\iota(s)}^s \left( |X_u^\phi - X_{\iota(s)}^\phi|^2 + |b(u, X_u^\phi, a_u^\phi)|^2 + \mathrm{Tr}\left(\sigma\sigma^\top(u, X_u^\phi, a_u^\phi)\right) \right) \mathrm{d}u 
\\
&\le \int_{\iota(s)}^s \mathbb{E}|X_u^\phi - X_{\iota(s)}^\phi|^2 \mathrm{d}u + M(s-\iota(s)),
\end{align*}
where $M:=\sup_{t \in [0,T],x\in Q,a\in A}(|b(t,x,a)|^2+\mathrm{Tr}(\sigma\sigma^\top(t,x,a)))$. By Gronwall's inequality, we conclude that $\mathbb{E}|X_{\iota(s)}^\phi - X_s^\phi|^2 \le C|\iota(s)-s| \leq C|\mathcal{G}|$
for some constant $C>0$ independent of $\phi$ and $\mathcal{G}$.

Thus, we have
\begin{align*}
\mathbb{E}|a_s^{\mathcal{G},\sigma_{\mathrm{exp}}}-a_s^\phi|^2&\le C F(s) + C\mathbb{E}|X_{\iota(s)}^\phi-X_s^\phi|^2 +C|\mathcal{G}|^2+C\sigma_{\mathrm{exp}}^2\mathbb{E}|\zeta_k|^2
\\
&\le C F(s)+C|\mathcal{G}|+C\sigma_{\mathrm{exp}}^2 M_\zeta.
\end{align*}
Substituting this bound into \eqref{eq:F-before-action} yields
\begin{align*}
F(t) \le C\int_0^tF(s)\mathrm{d}s+C\int_0^t\left(|\mathcal{G}|+\sigma_{\mathrm{exp}}^2M_\zeta
\right)\mathrm{d}s \le C\int_0^tF(s)\mathrm{d}s+C\left(|\mathcal{G}|+\sigma_{\mathrm{exp}}^2M_\zeta\right).
\end{align*}
By Gronwall's inequality, we have $F(T)\le C\left(|\mathcal{G}|+\sigma_{\mathrm{exp}}^2M_\zeta\right)$, which is precisely the desired error bound \eqref{eq:state-process-error}.
\end{proof}

\subsubsection{Objective error}
We now translate the state error bound into a bound on the difference between the ideal and the actually executed objective values. Recall that the ideal objective under the deterministic policy is
\begin{align*}
J(\phi):=\mathbb{E}\left[\int_0^T e^{-\beta t}r(t,X_t^\phi,a_t^\phi)\mathrm{d}t+e^{-\beta T}g(X_T^\phi)\right],
\end{align*}
while the objective achieved by the discretely sampled exploratory policy is
\begin{align*}
J_{\mathcal{G},\sigma_{\mathrm{exp}}}(\phi):=\mathbb{E}\left[\int_0^T e^{-\beta t}r(t,X_t^{\mathcal{G},\sigma_{\mathrm{exp}}},a_t^{\mathcal{G},\sigma_{\mathrm{exp}}})\mathrm{d}t+e^{-\beta T}g(X_T^{\mathcal{G},\sigma_{\mathrm{exp}}})\right].
\end{align*}

\begin{corollary}
\label{cor:objective-error}
Under the assumptions of Lemma \ref{prop:sampled-rsde-wellposed}, there exists a constant $C>0$, independent of $\mathcal{G}$ and $\sigma_{\mathrm{exp}}$, such that
\begin{align*}
|J_{\mathcal{G},\sigma_{\mathrm{exp}}}(\phi)-J(\phi)| \le C\left(|\mathcal{G}|^{1/2}+\sigma_{\mathrm{exp}}M_\zeta^{1/2}\right).
\end{align*}
\end{corollary}

\begin{proof}
By Lipschitz continuity of $r$ in $(x,a)$ and $g$ in $x$, it follows that
\begin{align*}
|J_{\mathcal{G},\sigma_{\mathrm{exp}}}(\phi)-J(\phi)|\le C\int_0^T\mathbb{E}\left[|X_t^{\mathcal{G},\sigma_{\mathrm{exp}}}-X_t^\phi|
+|a_t^{\mathcal{G},\sigma_{\mathrm{exp}}}-a_t^\phi|\right]\mathrm{d}t+C\mathbb{E}|X_T^{\mathcal{G},\sigma_{\mathrm{exp}}}-X_T^\phi|,
\end{align*}
for some constant $C>0$. Applying Theorem \ref{thm:state-process-error} gives 
\begin{align*} \mathbb{E}|X_t^{\mathcal{G},\sigma_{\mathrm{exp}}}-X_t^\phi| \le\left(\mathbb{E}|\Delta_t|^2\right)^{1/2}\le C\left(|\mathcal{G}|^{1/2}+\sigma_{\mathrm{exp}}M_\zeta^{1/2}\right).
\end{align*} 
Similarly, from the proof of Theorem \ref{thm:state-process-error}, we have
\begin{align*}
\mathbb{E}|a_t^{\mathcal{G},\sigma_{\mathrm{exp}}}-a_t^\phi|^2 \le CF(t)+C|\mathcal{G}|+C\sigma_{\mathrm{exp}}^2 M_\zeta \le C\left(|\mathcal{G}|+\sigma_{\mathrm{exp}}^2 M_\zeta\right),
\end{align*}
and hence $\mathbb{E}|a_t^{\mathcal{G},\sigma_{\mathrm{exp}}}-a_t^\phi| \le C(|\mathcal{G}|^{1/2}+\sigma_{\mathrm{exp}}M_\zeta^{1/2})$.
Combining these estimates yields the desired objective error bound.
\end{proof}

Corollary \ref{cor:objective-error} provides a quantitative guarantee that the performance of the discretely executed exploratory policy converges to that of the ideal deterministic policy as the time grid becomes finer ($|\mathcal G|\to0$) and the exploration noise vanishes ($\sigma_{\mathrm{exp}}\to 0$). The error bound reveals a trade-off inherent in the algorithm design: while exploration noise is necessary for learning the advantage-rate function in an off-policy manner, it introduces a bias in the objective relative to the ideal deterministic policy. Furthermore, this error analysis justifies the use of piecewise-constant action approximations in the implementation of the continuous-time algorithm. The bound assures us that the discrepancy between the theoretical ideal and the practical implementation is controlled by quantities that can be made arbitrarily small at the agent's discretion, which provides a rigorous foundation for the effectiveness of the CT-DDPG-R algorithm.

\section{Numerical Experiments}
\label{sec:application}

In this section, we execute the proposed CT-DDPG-R algorithm on a family of reflected reservoir control problems (\cite{pl75}, \cite{lai22}, \cite{luo25}). Reservoir systems provide a natural class of constrained continuous-time control tasks: the state variables represent water-level deviations from desired levels, and the admissible state region is a bounded box imposed by physical storage capacities and safety limits. We compare the soft constraint, hard constraint and the additional unconstrained case of the Neumann boundary condition in the value approximation.

\subsection{One-dimensional reservoir with an explicit optimal solution}
\label{subsec:1d}
We first consider a one-dimensional reservoir control problem for which the optimal value function and optimal policy admit an explicit analytical solution. This experiment serves as a valuable test for validating the proposed algorithm, as it allows us to compare the learned approximations against the ground truth.

Let $X_t\in[-H,H]$ denote the deviation of a reservoir level from its target level with $H>0$, and let $a_t\in[-A,A]$ be the regulation rate (i.e., the rate at which water is released or retained) with $A>0$. The reflected state dynamic is governed by
\begin{align*}
\mathrm{d}X_t=a_t\mathrm{d}t+\sigma\mathrm{d}W_t+\mathrm{d}K_t^- - \mathrm{d}K_t^+,\quad X_t\in[-H,H],
\end{align*}
where the local time processes $K^-$ and $K^+$ increase only when the process $X$ hits the lower boundary $-H$ and the upper boundary $H$ respectively. Our goal is to maximize the discounted reward
\begin{align*}
V^*(t,x)=\sup_{a\in{\mathcal{A}}} \mathbb{E}_{t,x}\left[\int_t^T e^{-\beta(s-t)} r(X_s,a_s)\mathrm{d}s+e^{-\beta(T-t)}g(X_T)\right].
\end{align*}
Here the running and terminal rewards  are given by $r(x)=-\eta x^2$, $g(x)=-\kappa x^2$ with $\eta,\kappa>0$, which penalize deviations from the target water level during operation and at the terminal time, and ${\mathcal{A}}$ is the admissible set of all $\mathbb{F}$-adapted processes $a=\{a_s\}_{s\in [0,T]}$ taking values in $[-A,A]$.

For this problem, the associated HJB equation takes the form
\begin{align*}
\begin{cases}
\displaystyle\partial_t V^*-\beta V^*+\frac{\sigma^2}{2}\partial_{xx}V^*-\eta x^2 +\sup_{a\in[-A,A]}\left(a\partial_xV^*\right)=0,&(t,x)\in[0,T)\times(-H,H),\\
\displaystyle \partial_xV^*(t,-H)=\partial_xV^*(t,H)=0,&t\in[0,T),\\
\displaystyle V^*(T,x)=-\kappa x^2,&x\in[-H,H].
\end{cases}
\end{align*}
By symmetry, $V^*(t,x)$ is even in $x$. Since the value decreases as $|x|$ increases, we have
$\partial_xV^*(t,x) \leq 0$ for $x>0$, and
$\partial_xV^*(t,x) \geq 0$ for $x<0$.
Therefore, the optimal policy has a bang-bang form:
\begin{align}\label{eq:1d-opt-policy}
\mu^*(t,x)=\argmax_{a\in[-A,A]} \left(a\partial_xV^*(t,x)\right)=-A\mathrm{sgn}(x).
\end{align}
Thus the optimal controller pushes the reservoir level toward the target level with maximal admissible intensity. Let $y=|x|$, $\tau=T-t$, and write $V^*(t,x)=-C(\tau,y)$. Under the bang-bang control \eqref{eq:1d-opt-policy}, the transformed function $C$ satisfies the following linear  parabolic equation on the half-interval $[0,H]$
\begin{align*}
\begin{cases}
\displaystyle \partial_\tau C = \frac{\sigma^2}{2}\partial_{yy}C-A\partial_yC -\beta C+\eta y^2,&(\tau,y)\in(0,T]\times(0,H),\\
\displaystyle \partial_yC(\tau,0)=\partial_yC(\tau,H)=0,&\tau\in(0,T],\\
\displaystyle C(0,y)=\kappa y^2,&y\in[0,H].
\end{cases}
\end{align*}
This equation is linear and can be solved explicitly via eigenfunction expansion. Define the constants $\alpha=\frac{A}{\sigma^2}$, $k_n=\frac{n\pi}{H}$ for $n\ge1$. The eigenfunctions of the operator $\mathcal{L} f =\frac{\sigma^2}{2}f''-A f'$ with Neumann boundary conditions $f'(0)=f'(H)=0$ are $\phi_0(y)=1$, and $\phi_n(y)=e^{\alpha y}[\cos(k_ny)-\frac{\alpha}{k_n}\sin(k_ny)]$ for $n\ge1$, with corresponding eigenvalues $\lambda_0=0$, $\lambda_n=\frac{A^2}{2\sigma^2}+\frac{\sigma^2k_n^2}{2}$ for $n\ge1$. Let
\begin{align*}
f_n=\frac{
\displaystyle\int_0^H e^{-2\alpha y}y^2\phi_n(y)\mathrm{d}y
}{
\displaystyle\int_0^H e^{-2\alpha y}\phi_n^2(y)\mathrm{d}y
},
\quad n\ge0.
\end{align*}
Then the expansion $y^2=\sum_{n\ge0}f_n\phi_n(y)$ holds, and the solution of the transformed PDE is given by
\begin{align*}
C(\tau,y)=\sum_{n=0}^{\infty}\left[e^{-(\beta+\lambda_n)\tau}\kappa f_n+\eta \frac{1-e^{-(\beta+\lambda_n)\tau}}{\beta+\lambda_n}f_n\right]\phi_n(y).
\end{align*}
Thus, by a standard verification argument, the optimal value function admits the explicit series representation
\begin{align}\label{eq:1d-vstar}
V^*(t,x)=-\sum_{n=0}^{\infty}\left[e^{-(\beta+\lambda_n)(T-t)}\kappa f_n+\eta\frac{1-e^{-(\beta+\lambda_n)(T-t)}}{\beta+\lambda_n}f_n\right]\phi_n(|x|).
\end{align}

In the numerical experiment, the infinite series in \eqref{eq:1d-vstar} is approximated by the first $150$ eigenfunctions, and the coefficients $f_n$ are computed by numerical quadrature on $4000$ grid points, ensuring sufficient accuracy for the ground truth. The parameters are $T=1,\beta=0.8,\sigma=0.3,H=1,A=1,\kappa=1,\eta=5$. The value network, advantage-rate network, and policy network are fully connected neural networks with three hidden layers, $64$ hidden units per layer and hyperbolic tangent activations. The actor output is bounded by $A\tanh(\cdot)$. The hard value network uses the cosine frequencies $\{1,2,3,4,6,10,14\}$ for the rescaled state variable. We train for $3000$ episodes with $\Delta t=0.01$ with Gaussian exploration noise decaying from $0.5$ to $0.2$. The deterministic return is evaluated every $50$ episodes using $100$ rollout paths.

Figure \ref{fig:1d-returns-residual} reports the returns without exploration and the boundary residual during the updating process. We can see a clear separation among the three boundary treatments in terms of the Neumann boundary residual which confirms that both boundary-aware treatments effectively encode the structure of the value network. It is worth noting that the unconstrained neural networks, due to strong expression power, are still able to recover the value function and the optimal bang-bang policy with reasonable accuracy; see Figures \ref{fig:1d-value-policy}. Nevertheless, we find that the unconstrained training process exhibits more pronounced oscillations and occasional degradation. In contrast, the soft and hard boundary treatments provide additional structural information and therefore lead to more stable learning dynamics in many runs.

At the same time, this experiment also shows a limitation of imposing the boundary condition that the learning process may quickly drive the policy toward the boundary and remain trapped there for a long period of training, especially under the hard constraint. This issue is no longer observed in the subsequent reservoir experiments, where the running reward contains a logarithmic boundary penalty. In favor of the boundary penalty, as the dimension of the state space increases and the value function and policy become more complex, the benefit of incorporating the Neumann constraints becomes more visible and they lead to faster and more stable convergence than the unconstrained baseline.

\begin{figure}[h]
    \centering
    \begin{subfigure}[b]{.45\linewidth}
        \centering
        \includegraphics[scale=0.45]{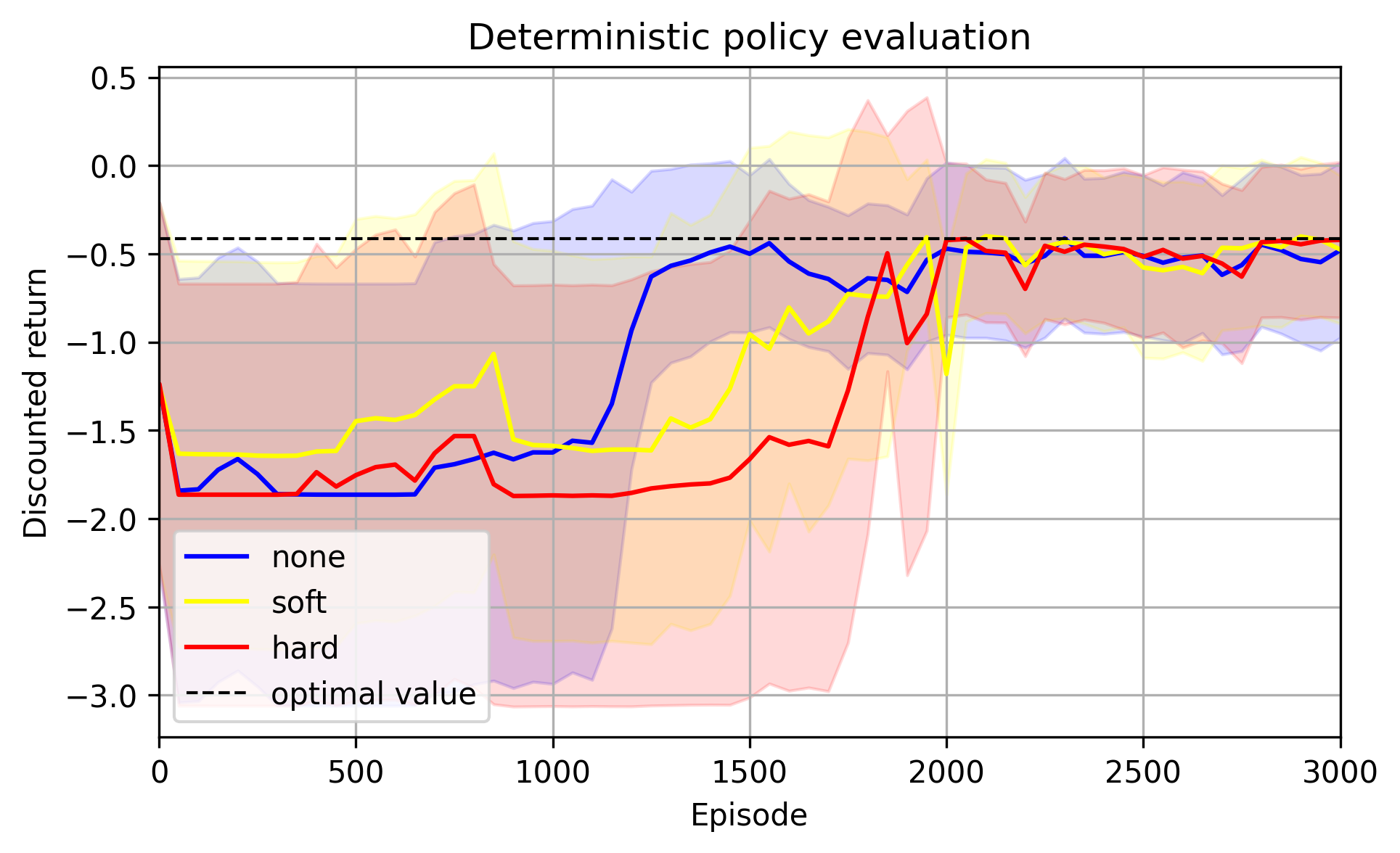}
    \end{subfigure}
    \begin{subfigure}[b]{.45\linewidth}
        \centering
        \includegraphics[scale=0.45]{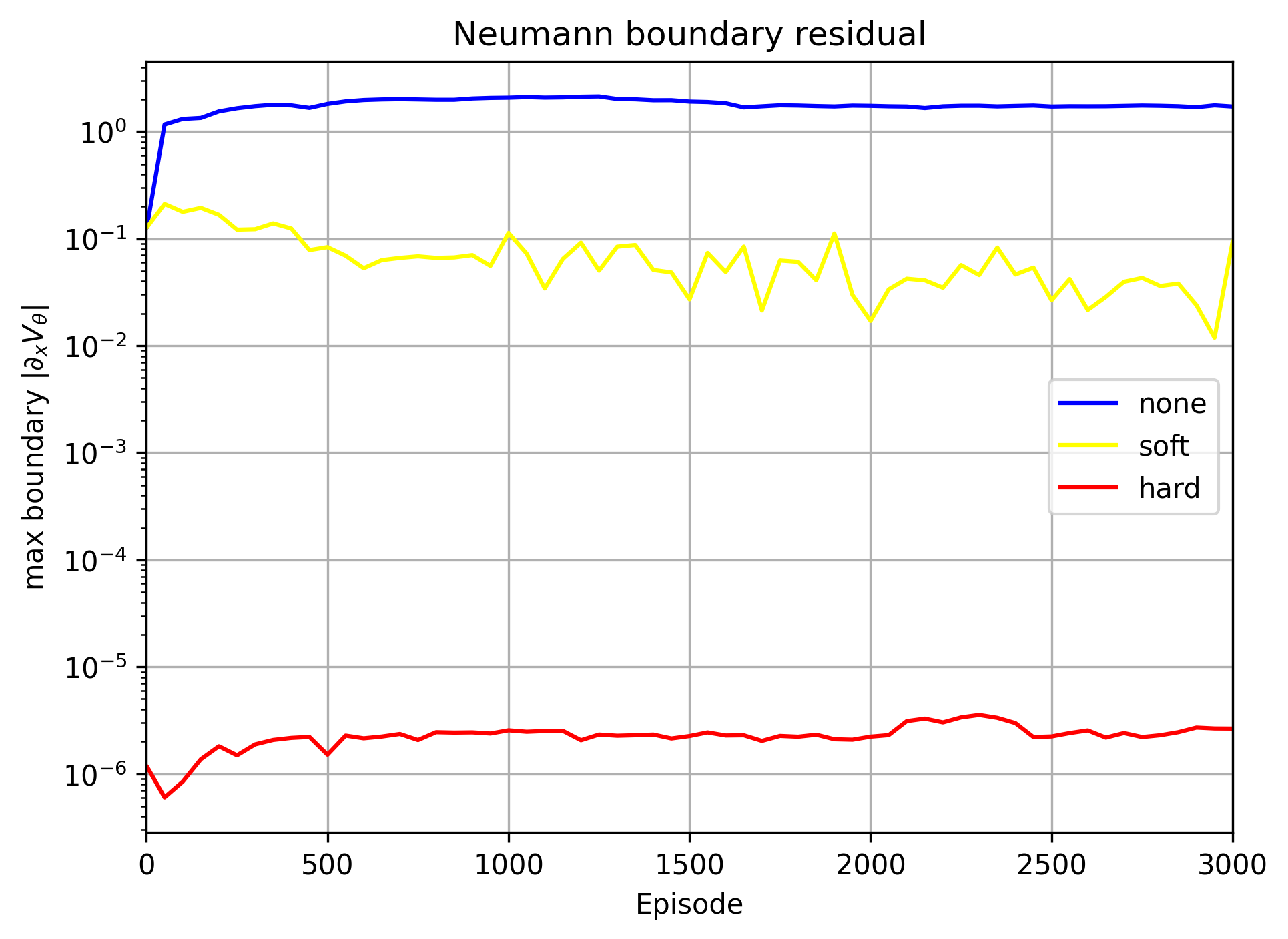}
    \end{subfigure}
    \caption{Returns and boundary residual}
    \label{fig:1d-returns-residual}
\end{figure}

\begin{figure}[h]
    \centering
    \begin{subfigure}[b]{.3\linewidth}
        \centering
        \includegraphics[width=\linewidth]{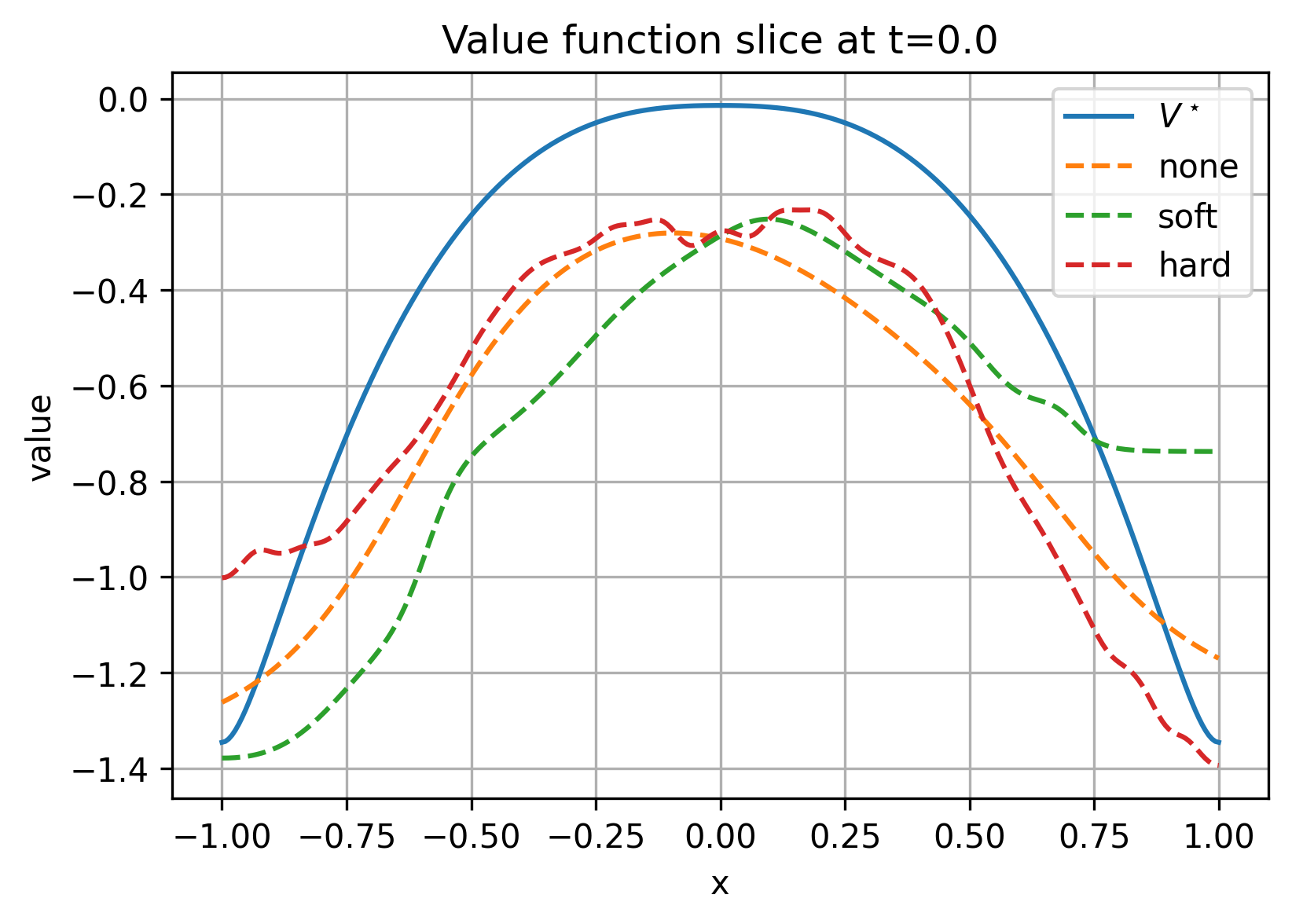}
    \end{subfigure}
    \begin{subfigure}[b]{.3\linewidth}
        \centering
        \includegraphics[width=\linewidth]{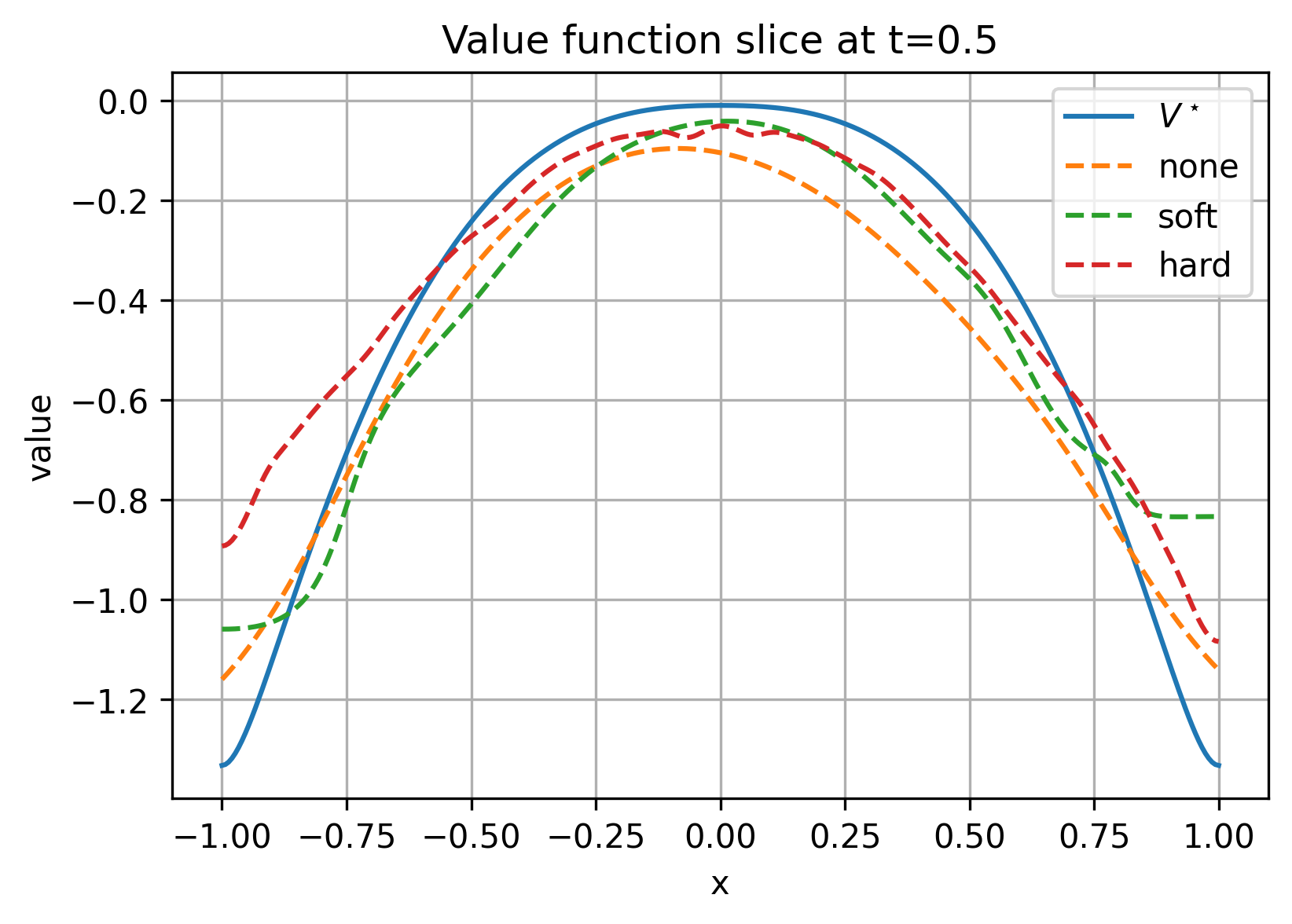}
    \end{subfigure}
    \begin{subfigure}[b]{.3\linewidth}
        \centering
        \includegraphics[width=\linewidth]{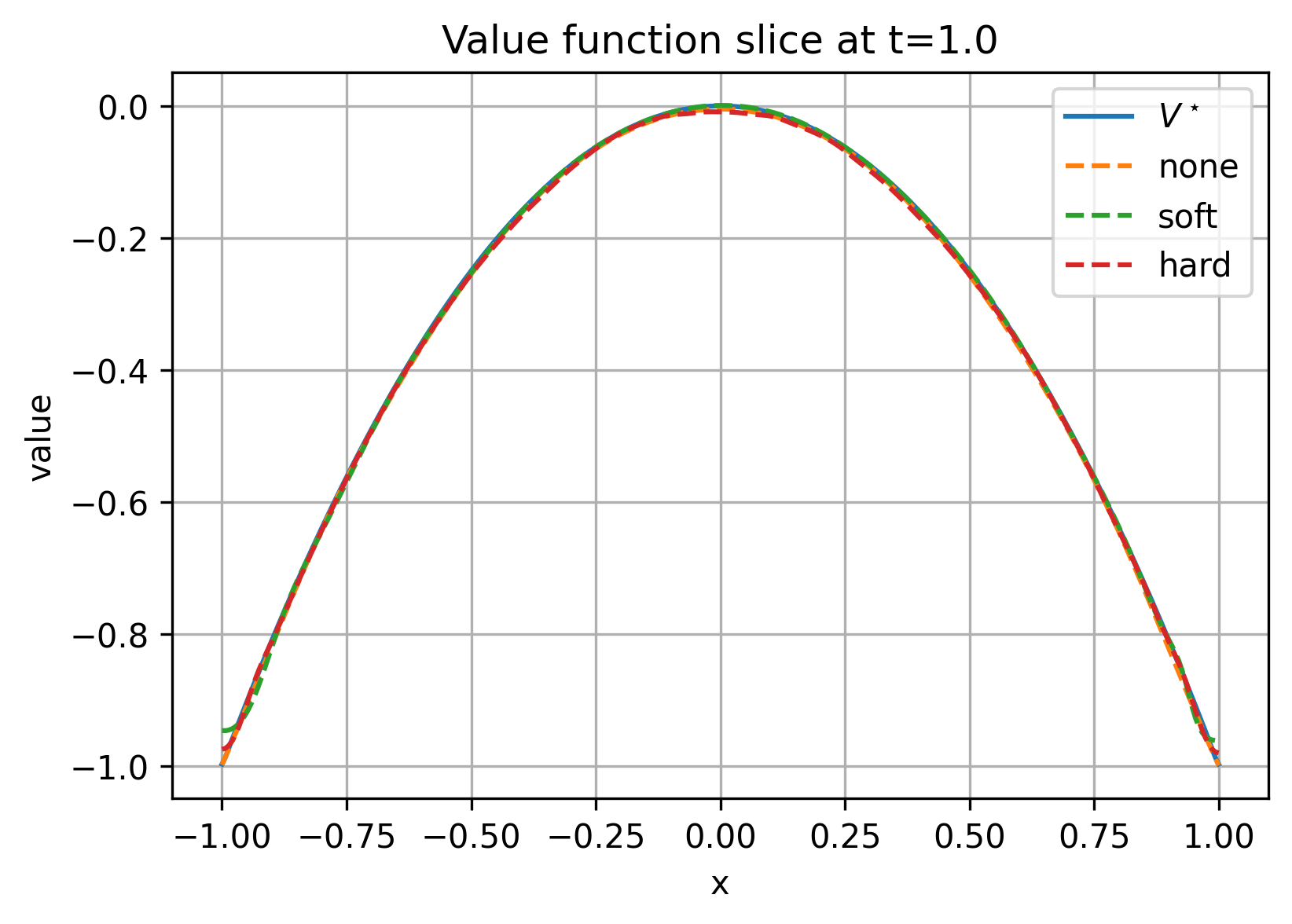}
    \end{subfigure}
    
    \vspace{1ex} 
    
    \begin{subfigure}[b]{.3\linewidth}
        \centering
        \includegraphics[width=\linewidth]{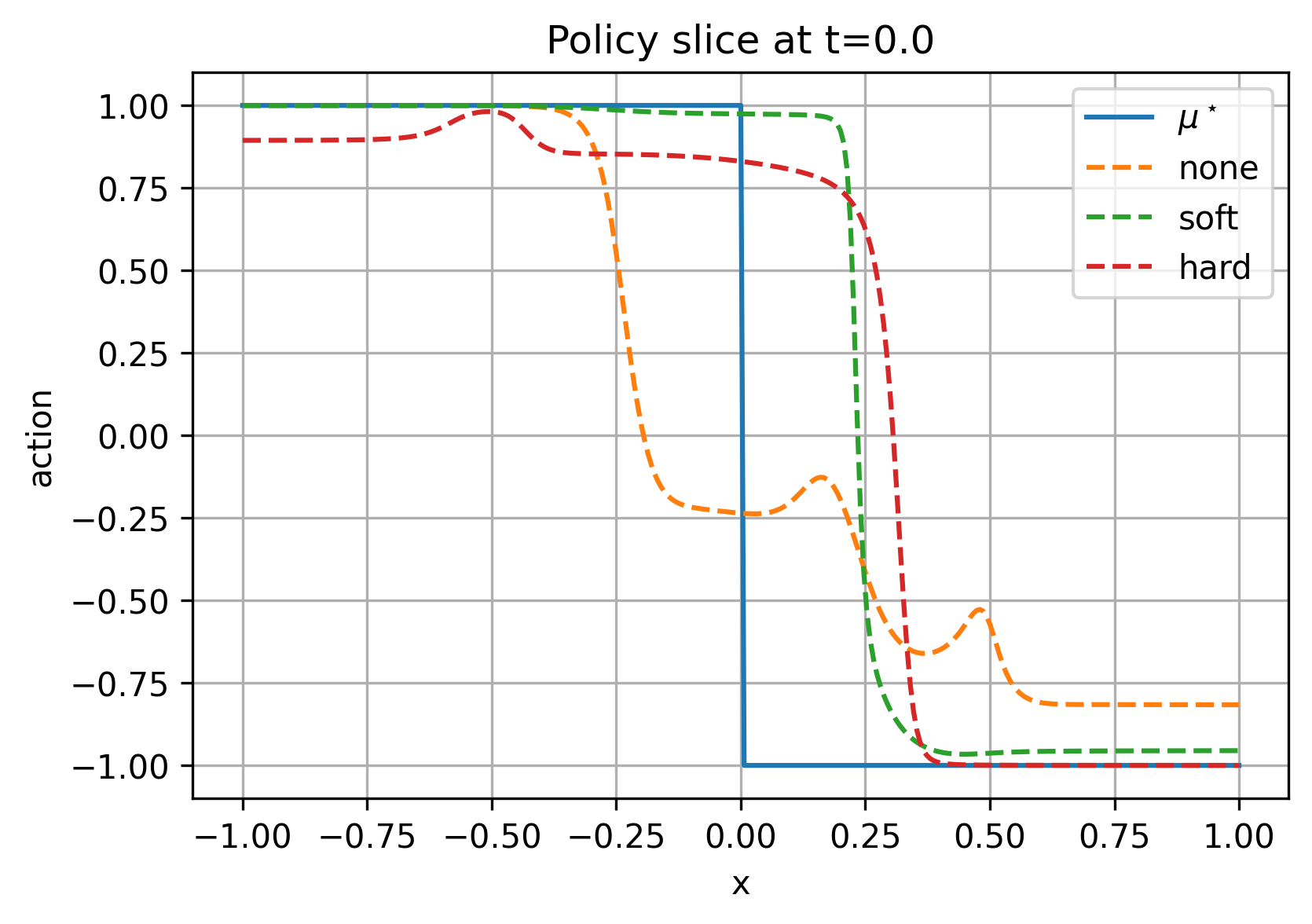}
    \end{subfigure}
    \begin{subfigure}[b]{.3\linewidth}
        \centering
        \includegraphics[width=\linewidth]{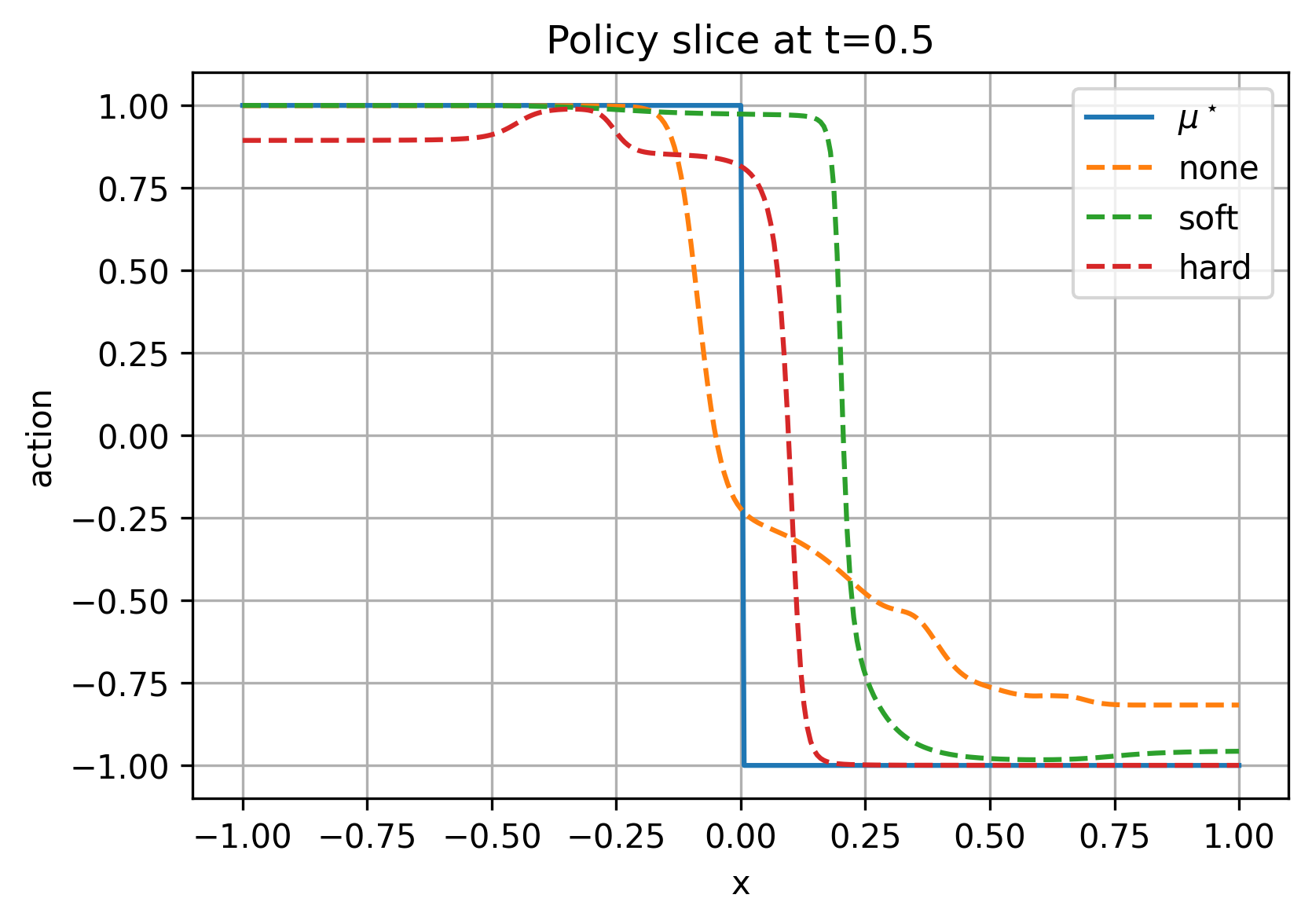}
    \end{subfigure}
    \begin{subfigure}[b]{.3\linewidth}
        \centering
        \includegraphics[width=\linewidth]{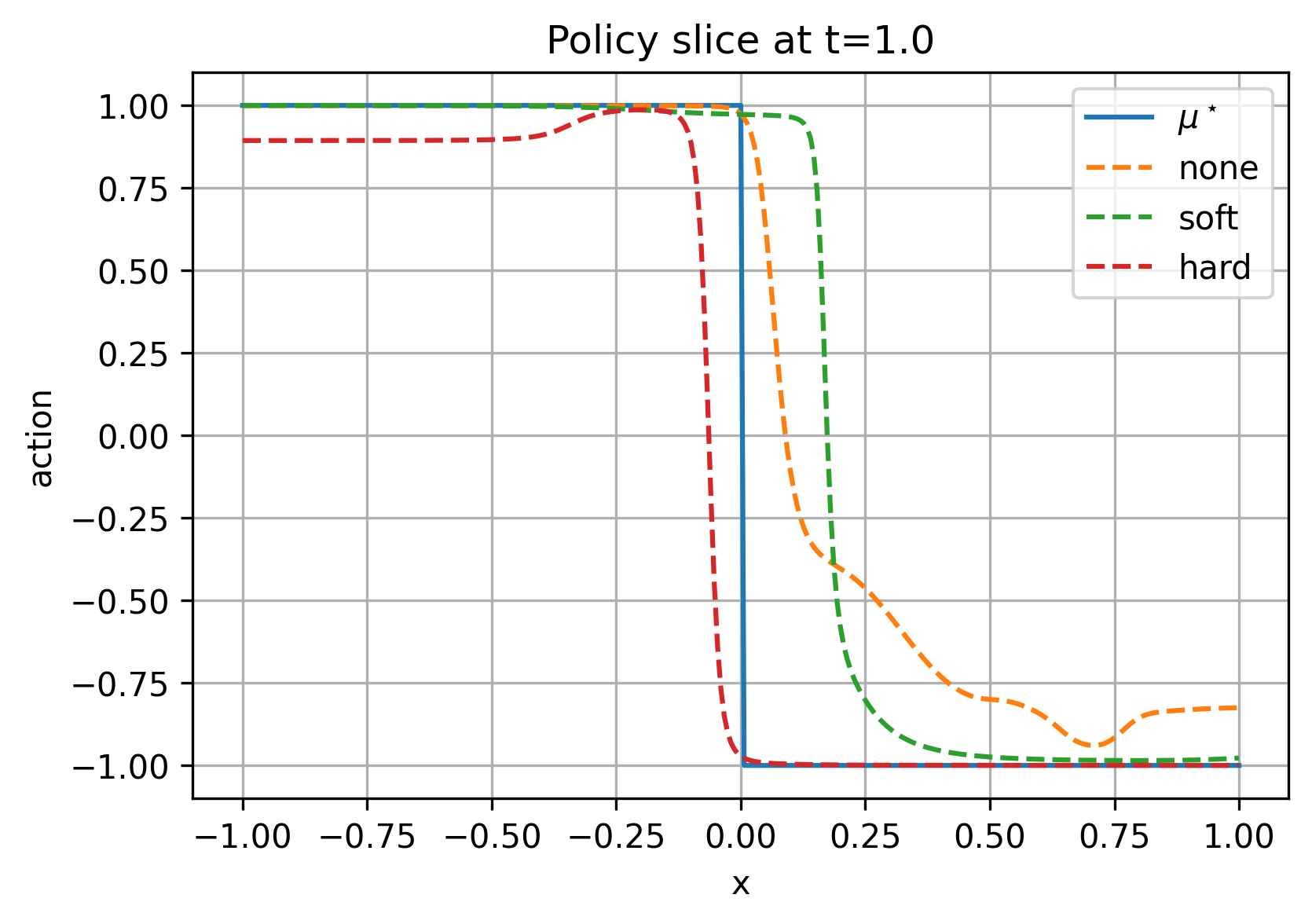}
    \end{subfigure}
    
    \caption{Value functions and policies under different constraints}
    \label{fig:1d-value-policy}
\end{figure}

\subsection{Two-dimensional reservoir cascade regulation}
\label{subsec:2d}

We next consider a two-dimensional reservoir cascade. Let $X_t^i\in[-H_i,H_i]$ be the water-level deviation of reservoir $i$, and let $a_t^i\in[-A_i,A_i]$ be the relative release rate of reservoir $i$. A larger release from reservoir $1$ reduces the upstream level but increases the downstream inflow into reservoir $2$. The reflected dynamics are governed by the following system of RSDEs:
\begin{align*}
\mathrm{d}X_t^1&=\left(m_1(t)-a_t^1\right)\mathrm{d}t+\sigma_1\mathrm{d}W_t^1+\mathrm{d}K_t^{-,1}-\mathrm{d}K_t^{+,1},\\
\mathrm{d}X_t^2&=\left(m_2(t)+\rho a_t^1-a_t^2\right)\mathrm{d}t+\sigma_2\mathrm{d}W_t^2+\mathrm{d}K_t^{-,2}-\mathrm{d}K_t^{+,2}.
\end{align*}
Here $m_i(t)$ is the mean net inflow deviation of reservoir $i$,  which may vary over time due to seasonal or stochastic factors, and $\rho\in[0,1]$ is the downstream transfer coefficient. The term $\rho a_t^1$ models the fraction of the upstream release that enters the downstream reservoir. The reflection terms $K^{\pm,i}$ ensure that the state process remains within the admissible hypercube $[-H_1,H_1]\times[-H_2,H_2]$ by enforcing normal reflection at the boundaries.

The running reward is
\begin{align*}
r(t,x,a)=- \Bigg[&q_1x_1^2+q_2x_2^2+q_{12}\left(\frac{x_1}{H_1}-\frac{x_2}{H_2}\right)^2+\lambda_1a_1^2+\lambda_2a_2^2 \\
&\quad+\gamma_{b}\left[-\log\left(1-\left(\frac{x_1}{H_1}\right)^2+\varepsilon_{b}\right)-\log\left(1-\left(\frac{x_2}{H_2}\right)^2+\varepsilon_{b}\right)\right]\Bigg].
\end{align*}
The first two terms penalize deviations of each reservoir from its target level, encouraging the system to maintain desired storage volumes. The third term penalizes imbalance between the relative storage levels of the two reservoirs, promoting an equitable distribution of water across the cascade. The fourth and fifth terms penalize aggressive control which causes operational costs. The logarithmic terms serve as barrier functions that increasingly penalize states approaching the boundary of the admissible domain, thereby discouraging the system from operating near the physical limits. The terminal reward is analogously defined as
\begin{align*}
g(x)=-\left[\kappa_1x_1^2+\kappa_2x_2^2+\kappa_{12}\left(\frac{x_1}{H_1}-\frac{x_2}{H_2}\right)^2\right].
\end{align*}

In this experiment, the parameters are set as $T=1$, $\beta=0.5$, $\Delta t=0.01$, $H_1=H_2=A_1=A_2=1$, $\sigma_1=\sigma_2=0.45$, $\rho=0.65$, $m_1(t)=0.25\sin(2\pi t/T)$, $m_2(t)=0.20\cos(2\pi t/T)$, $q_1=q_2=1$, $q_{12}=0.8$, $\lambda_1=\lambda_2=0.03$, $\kappa_1=\kappa_2=1$, $\kappa_{12}=0.8$, $\gamma_{b}=0.03$, and $\varepsilon_{b}=10^{-3}$. The value, advantage-rate, and actor networks are fully connected networks with three hidden layers, $128$ hidden units per layer, and $\tanh$ activations. For the hard Neumann method, the value network replaces the raw state input by cosine features with frequencies $\{1,2,3,4,6,8\}$ in each coordinate. The actor is bounded componentwise by $(A_1,A_2)\cdot\tanh(\cdot)$. We train for $4000$ episodes and the returns during training can be seen in Figure \ref{fig:2d-returns} with five times the standard deviation for visual clarity. Since no closed-form solution is available for this two-dimensional problem, we evaluate the learned policies qualitatively by visualizing their surfaces at selected time slice. Figures \ref{fig:2d-policy-none}, \ref{fig:2d-policy-soft}, and \ref{fig:2d-policy-hard} display the learned policy functions for the unconstrained, soft-constrained, and hard-constrained approaches, respectively.

\begin{figure}[ht]
    \centering
    \includegraphics[scale=0.5]{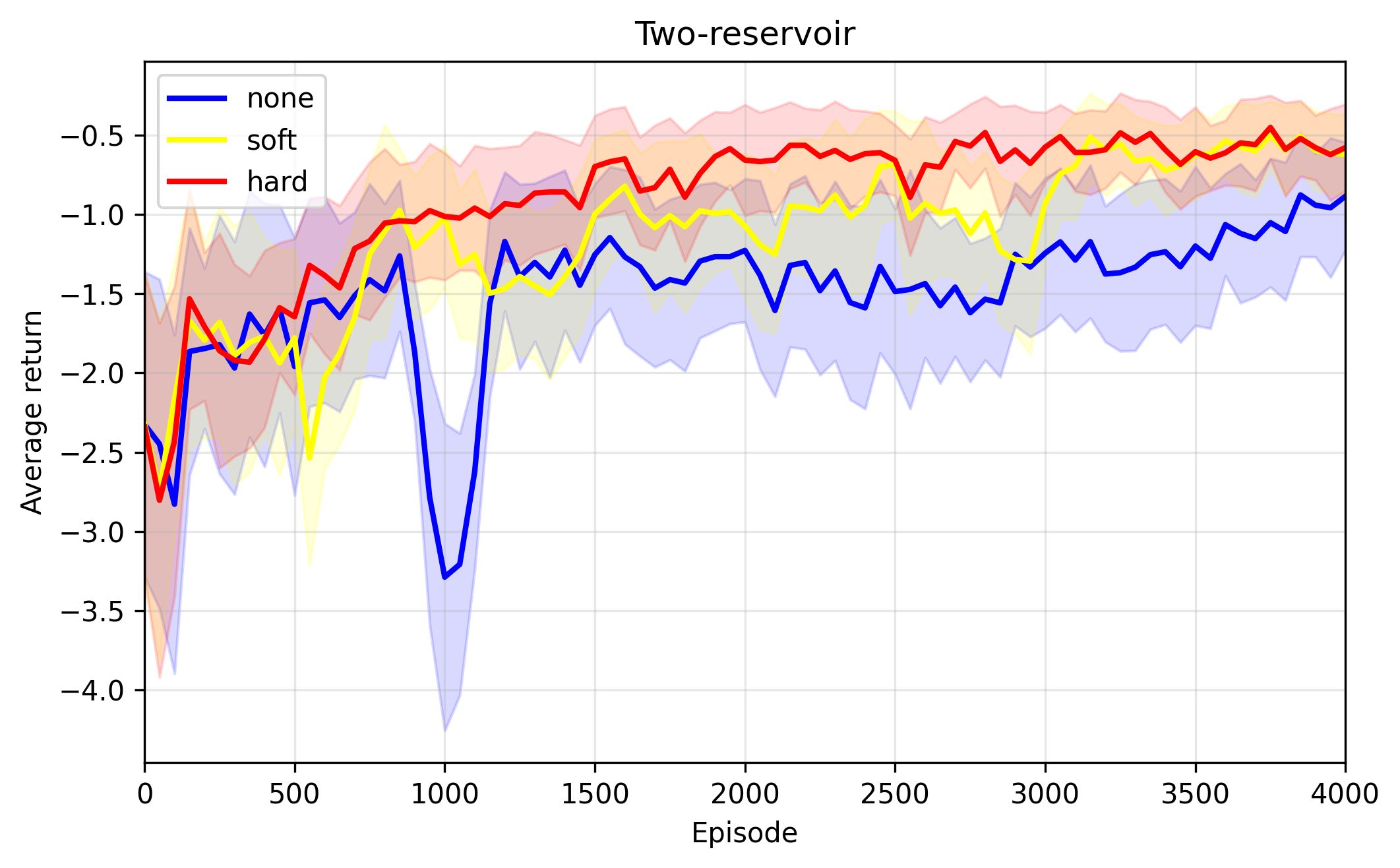}
    \caption{Returns during training}
    \label{fig:2d-returns}
\end{figure}

To provide a quantitative reference, we also compute a finite-difference approximation of the optimal policy by numerically solving the associated HJB equation using the method of \cite{kush01}; the resulting reference policy is shown in Figure \ref{fig:2d-policy-fd}. While the finite-difference reference exhibits some numerical oscillations owing to the limitations of grid-based discretization in two dimensions, the overall qualitative structure is broadly consistent across all methods.  Notably, the agreement with the reference policy becomes more apparent when the Neumann boundary condition is incorporated through either the soft or hard constraint, suggesting that boundary awareness helps the neural network capture the correct boundary behavior. These results further support the conclusion that the proposed boundary-aware treatments effectively encode the structural information which is helpful for stable and accurate learning in multidimensional settings. 

\begin{figure}[htbp]
    
    
        \centering
        \includegraphics[width=0.7\linewidth]{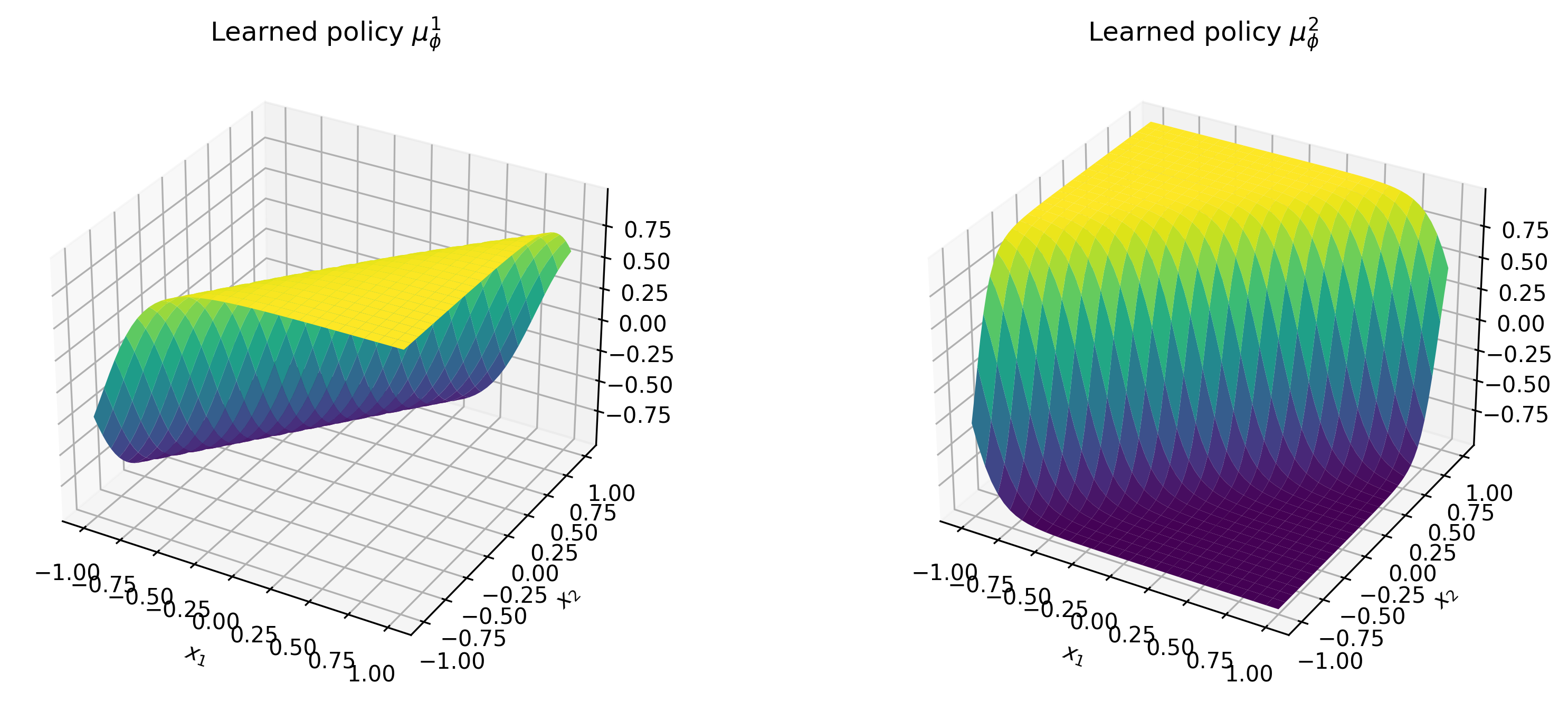}


    \caption{Policy surfaces without boundary constraint at $t=0.5$.}
    \label{fig:2d-policy-none}
\end{figure}

\begin{figure}[htbp]
    \centering
    
    
        \includegraphics[width=0.7\linewidth]{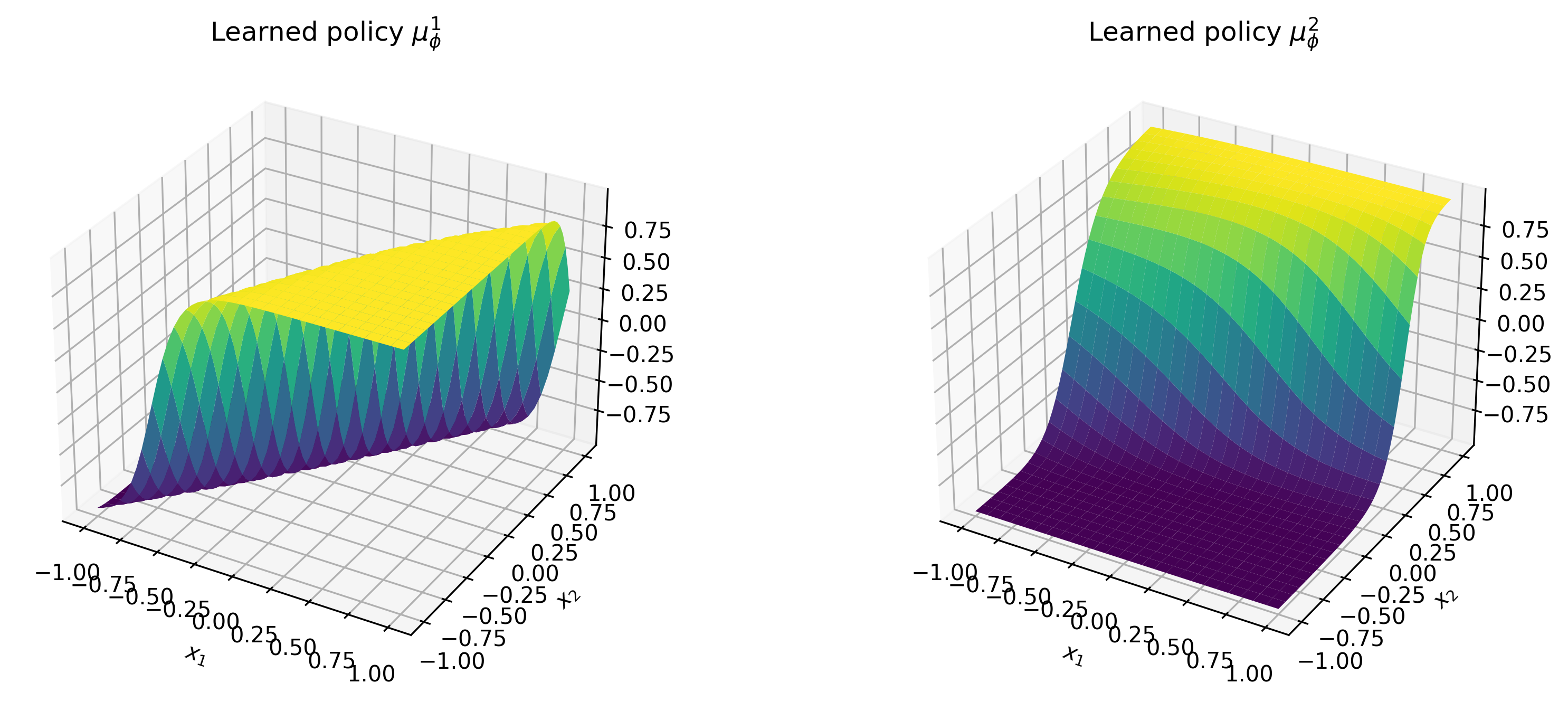}


    \caption{Policy surfaces under soft constraint at $t=0.5$.}
    \label{fig:2d-policy-soft}
\end{figure}

\begin{figure}[ht]
    \centering
    
    
        \includegraphics[width=0.7\linewidth]{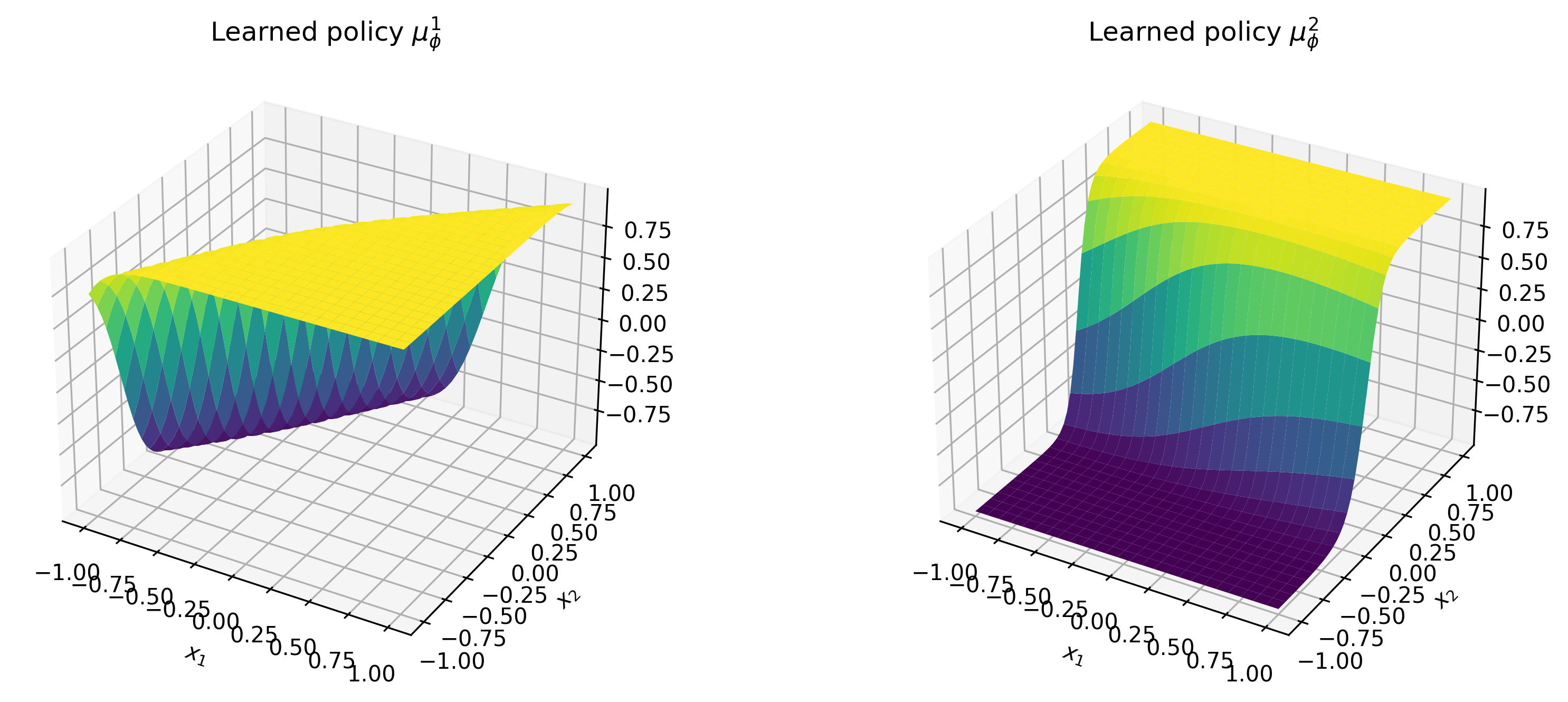}


    \caption{Policy surfaces under hard constraint at $t=0.5$.}
    \label{fig:2d-policy-hard}
\end{figure}

\begin{figure}[ht]
    \centering
    
    
        \includegraphics[width=0.7\linewidth]{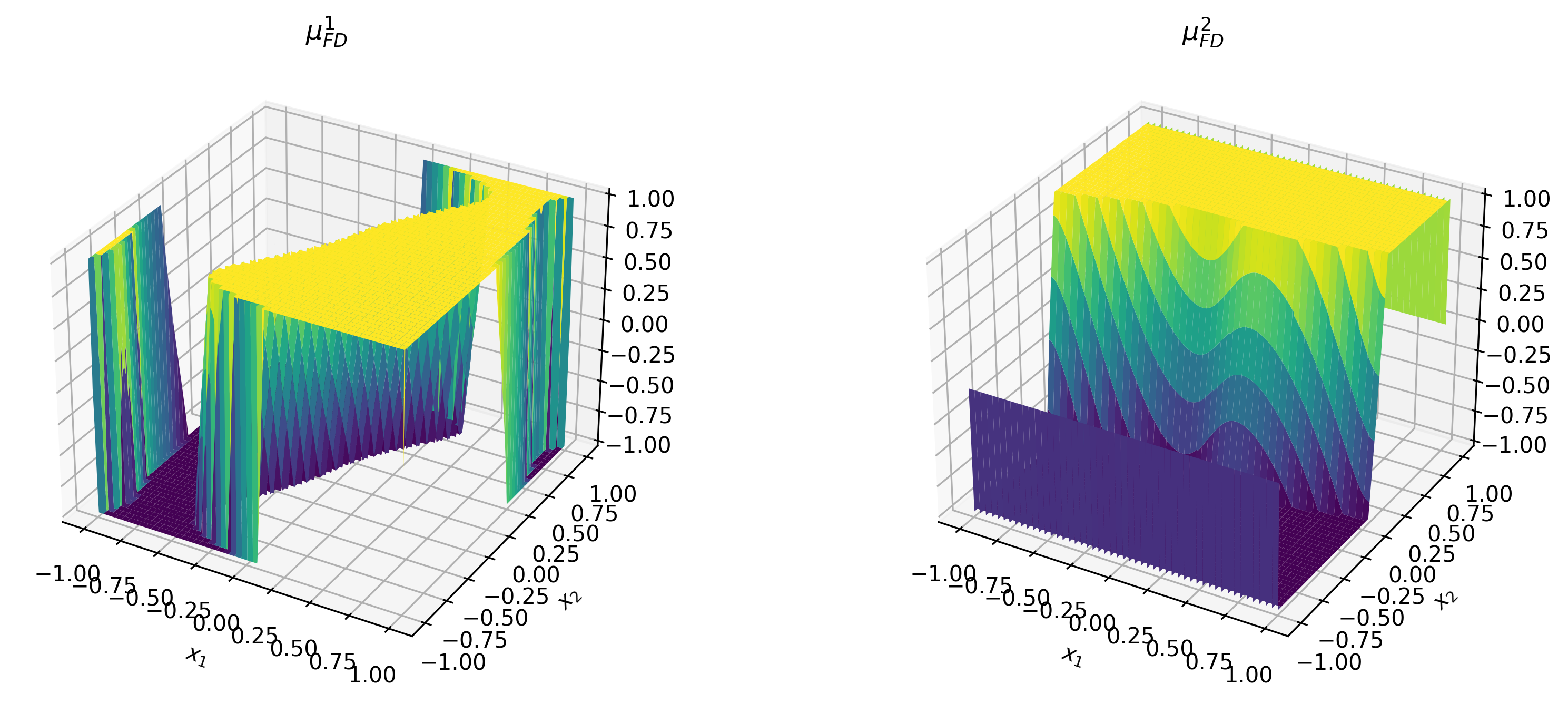}


    \caption{Finite-difference reference policy at $t=0.5$.}
    \label{fig:2d-policy-fd}
\end{figure}

\subsection{Five-dimensional reservoir cascade regulation}
\label{subsec:5d}

Finally, to assess the scalability of the proposed method to higher-dimensional settings, we consider a five-dimensional reservoir cascade. The state vector is given by
\begin{align*}
X_t=(X_t^1,\ldots,X_t^5)\in Q:=\prod_{i=1}^5[-H_i,H_i],
\end{align*}
where $X_t^i$ is the water-level deviation of reservoir $i$ from its target level. The control
$a_t=(a_t^1,\ldots,a_t^5)\in\prod_{i=1}^5[-A_i,A_i]$ represents the release rates. The reflected diffusion dynamics are governed by
\begin{align*}
\mathrm{d}X_t=b(t,X_t,a_t)\mathrm{d}t+	\mathrm{diag}(\sigma_1,\ldots,\sigma_5)\mathrm{d}W_t+\mathrm{d}K_t^- -\mathrm{d}K_t^+.
\end{align*}
The drift structure remains the same as in the two-dimensional cascade, with each reservoir receiving inflow from its upstream neighbor:
$b_1(t,x,a)=m_1(t)-a_1$ and $b_i(t,x,a)=m_i(t)+\rho_{i-1}a_{i-1}-a_i$ for $i=2,\ldots,5$. We set the net inflow deviations to $m_i(t)=\bar{m}_i\sin(2\pi t/T+\varphi_i)$, for $i=1,\ldots,5$.

The running reward is designed to capture the same operational objectives as in the lower-dimensional cases, now extended to five reservoirs:
\begin{align*}
r(t,x,a)=-\sum_{i=1}^5\left[(q_ix_i^2+\lambda_i a_i^2)-\gamma_{b}\log\left(1-\left(\frac{x_i}{H_i}\right)^2+\varepsilon_{b}\right)\right]-q_{adj}\sum_{i=1}^{4}\left(\frac{x_i}{H_i}-\frac{x_{i+1}}{H_{i+1}}\right)^2,
\end{align*}
and the terminal reward is analogously defined as
\begin{align*}
g(x)=-\left[\sum_{i=1}^5\kappa_ix_i^2+\kappa_{adj}\sum_{i=1}^{4}\left(\frac{x_i}{H_i}-\frac{x_{i+1}}{H_{i+1}}\right)^2\right].
\end{align*}

In the experiment, we choose the parameters as $T=1$, $\beta=0.5$, $\Delta t=0.01$, $H_i=A_i=q_i=\kappa_i=1$, $\lambda_i=0.03$, for $i=1,\ldots,5$, $(\sigma_1,\ldots,\sigma_5)=(0.45,0.43,0.40,0.38,0.35)$, $(\rho_1,\rho_2,\rho_3,\rho_4)=(0.65,0.55,0.45,0.35)$, $(\bar m_1,\ldots,\bar m_5)=(0.25,0.22,0.20,0.18,0.16)$, $(\varphi_1,\ldots,\varphi_5)=\left(0,\frac{\pi}{2},\pi,\frac{3\pi}{2},\frac{\pi}{4}\right)$ and $q_{adj}=\kappa_{adj}=0.8$, $\gamma_{b}=0.03$, $\varepsilon_{b}=10^{-3}$.
The value, advantage-rate, and actor networks use five hidden layers with $128$ hidden units per layer and $\tanh$ activations. The hard Neumann value network uses cosine frequencies $\{1,2,3,4,6,8,10,12\}$ in each state coordinate. We train for $4000$ episodes and display the returns in Figure \ref{fig:5d-returns} with five times the standard deviation for visual clarity. The same as in the two-dimensional experiment, the results here demonstrate that both the soft and hard boundary treatments lead to more stable and faster convergence than the unconstrained baseline. The boundary-aware methods achieve slightly better final performance, as showed in the return curves, and the learned policy surfaces exhibit smoother and more physically plausible behavior. This observation reinforces the central message of this paper: incorporating structural information about the Neumann boundary condition into the value network, whether through a soft penalty or a hard architectural constraint, provides significant benefits in terms of learning stability and final performance, particularly as the state dimension increases and the control problem becomes more complex.

\begin{figure}[htbp]
    \centering
    \includegraphics[scale=0.5]{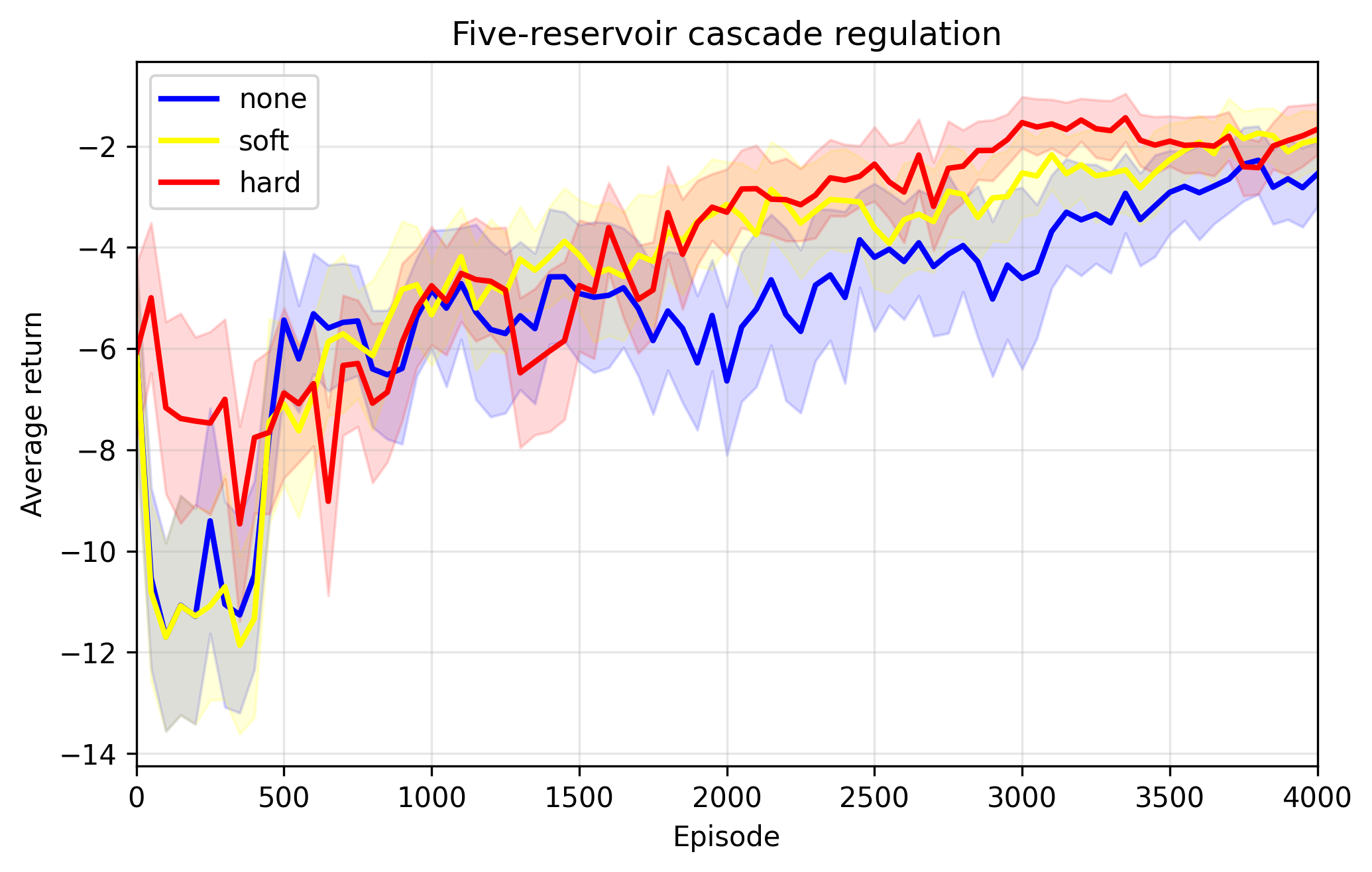}
    \caption{Returns during training.}
    \label{fig:5d-returns}
\end{figure}


\end{document}